\documentclass[]{interact}

\usepackage{epstopdf}
\usepackage[caption=false]{subfig}

\usepackage[numbers,sort&compress]{natbib}
\bibpunct[, ]{[}{]}{,}{n}{,}{,}
\renewcommand\bibfont{\fontsize{10}{12}\selectfont}

\theoremstyle{plain}
\newtheorem{theorem}{Theorem}[section]
\newtheorem{lemma}[theorem]{Lemma}

\theoremstyle{definition}

\theoremstyle{remark}

\usepackage{amsmath,amssymb,graphicx}
\usepackage{setspace}
\usepackage{float}
\usepackage[usenames,dvipsnames]{color}

\numberwithin{equation}{section}

\definecolor{Orange}{cmyk}{0, 0.6, 0.8, 0}

\newtheorem{thm}{Theorem}[section]
\newtheorem{prop}[thm]{Proposition}

\newtheorem{defn}[thm]{Definition}

\newtheorem{conj}[thm]{Conjecture}

\newtheorem{claim}[thm]{Claim}

\usepackage{fullpage}
\usepackage{xcolor}

\begin{document}

\title{Long-term Averages of the Stochastic Logistic Map}



\author{
\name{Maricela Cruz\textsuperscript{a} and Austin Wei\textsuperscript{b}  and Johanna Hardin\textsuperscript{c}  and Ami Radunskaya\textsuperscript{c}\thanks{CONTACT Ami Radunskaya. Email: Ami.Radunskaya@pomona.edu}}
\affil{\textsuperscript{a}Kaiser Permanente Washington Health Research Institute, Seattle WA; \textsuperscript{b}The Ohio State University, Columbus OH; \textsuperscript{c}Dept. of Mathematics and Statistics, Pomona College, Claremont CA}}

\maketitle

\doublespacing

\linespread{1.5}

\begin{abstract}
The logistic map is a nonlinear difference equation well studied in the literature, used to model self-limiting growth in certain populations. It is known that, under certain regularity conditions, the stochastic logistic map, where the parameter is varied according to a specified distribution, has a unique invariant distribution. In these cases we can compare the long-term behavior of the stochastic system with that of the deterministic system evaluated at the average parameter value. Here  we examine the relationship between the mean of the stochastic logistic equation and the mean of orbits of the deterministic logistic equation at the expected value of the parameter. We formally prove that, in some cases, the addition of noise is beneficial to the populations, in the sense that it increases the mean, while for other ranges of parameters it is detrimental. A conjecture  based on numerical evidence is  presented at the end.
\end{abstract}

\begin{keywords}
difference equation; invariant distribution; random dynamical system; long-term average
\end{keywords}

\newpage

\section{Introduction}

The logistic map is a nonlinear first-order difference equation first used to describe the evolution of biological populations. It was formulated by Pierre Francois Verhulst in $1838$ \cite{verhulst1838} but became a phenomenon in the $1940$'s when John Von Neumann and Stanislav Ulam suggested its use as a random number generator \cite{Ulam1947}. Since the seminal work of Robert May in 1976 \cite{May}, the logistic model has become a canonical equation describing self-limiting growth, and in the past fifty years it has been used extensively in both its discrete and continuous versions.  Our understanding of the behavior of the logistic map can be extended to    more general models, and, despite its simplicity, the logistic model gives rise to much of the qualitative behavior observed in population data, capturing both the stable and chaotic aspects of growth. For example, many protozoan populations such as yeast and grain beetles illustrate logistic growth in laboratory studies, exhibiting both stable and chaotic behavior \cite{Ecol}.  In the 1980's, economists used logistic models and their relatives to show that, surprisingly at the time, chaotic behavior could occur in models of capital accumulation \cite{Boldrin1986, Deneckere1986}.

If the environment or growth rate is noisy, it is useful to use a logistic model that includes either additive or parametric noise. For example, when making decisions in noisy environments, we might choose to use a stochastic model.    Nishimura, Stachurski, et al. give a survey of discrete time decision models, both deterministic and stochastic, in which the goal is to optimize savings and consumption \cite{Nishimura2004}.  They note the importance of understanding the existence and uniqueness of invariant measures in the stochastic case, a topic which we address in this paper. Many other forecasting problems can be formulated in terms of discrete stochastic maps \cite{Satoh2001}; for example, in \cite{Satoh2021}, Satoh uses discrete stochastic logistic and Gompertz models to predict when software is reliable enough to be released.  In related work,  Erguler and Stumpf  used the stochastic logistic map to contrast the ``true'' dynamics of a system and observational noise, thereby addressing the reliability of inference from present data \cite{Kamil}. The stochastic logistic model fits experimental data as well.  For example, in \cite{Descheemaeker2020}  predictions of the abundance of microbial communities using stochastic generalized Lotka-Volterra models provide good fits to experimental models.

In the study of the deterministic logistic map the focus is on the existence and stability of fixed points and periodic orbits.  In the stochastic case, invariant measures, or distributions, take the place of periodic points. If these invariant measures are stable, then they can be used to predict long-term behavior.
Diaconis and Freeman \cite{Diaconis1999} develop a ``contraction in the log average" criterion for the existence of a stable distribution of Markov processes.  The idea of a stable distribution, i.e., a stable fixed point of the Foias operator described in Section \ref{sec:ergodicity}, can be studied in both deterministic and stochastic maps.  One standard example is the tent map; the two-dimensional tent map has been explored in this context by Nakamura and Mackey in \cite{Nakamura2021}.  
Arethreya and Dai have studied the stochastic logistic map \cite{Athreya2000} and have found situations in which there are multiple invariant measures, so that the long-term behavior of the system depends on the initial state \cite{Athreya2002}.  For another example of conditions that ensure a unique invariant measure for stochastic one-dimensional maps, see Nakamura's study of non-expanding piecewise linear maps which can exhibit stable periodic behavior even in the presence of noise \cite{Nakamura2018}.  For a compendium of recent results and some new formulations, the reader is referred to \cite{Bhattacharya2007, Lasota}.

In making predictions using stochastic models, a natural question is: ``How do these predictions change as the amplitude of the noise increases?"  
In this paper we present a novel result that describes changes in the long-term average of the logistic map with multiplicative noise.  In other words, we study the logistic map when the parameter randomly fluctuates in a prescribed range, and answer the question: ``Does the presence of noise increase or decrease the long-term average?"  It turns out that the answer to this question depends on the periodicity of the stable orbits of the underlying deterministic system.  Our theoretical results can be applied to a range of predictive stochastic logistic models. 

Our paper contributes to the field in two distinct ways.  As far as we know, we are the first to recognize that the impact of noise on the deterministic system is a function of the periodicity of the stable orbits of that system. Indeed, the impact of the periodicity {\bf reverses} for each period doubling, for those periods we have studied. Using a known result (which is only relevant in the stable period one region) we show (Theorem \ref{PeriodOneTheorem}) that the stochastic long-term average is smaller than the deterministic long-term average in the stable period one region.  Our second big contribution is to formally prove that the stochastic long-term average is greater than the deterministic long-term average in the period two region; this is the content of Theorem \ref{PeriodTwoTheorem}. We further develop our understanding of the switching behavior with extended simulations in periods one, two, and four.

The paper is organized as follows: in Section \ref{sec:background} we present notation and general results on the stochastic logistic map.  The main results of the paper are found in Section \ref{sec:simulations}, where we prove that the long-term average increases in the period-one regime and decreases in the period-two regime.  The proof of the period-two result is quite technical, and the details are given in Section \ref{sec:appendix}.  A discussion section follows in Appendix \ref{sec:discussion}, where we present a more general ``flip-flop" conjecture.

\section{Background}
\label{sec:background}

\subsection{The Map}

The logistic map is a one-parameter family which gives the difference equation:

$$ x_{n+1}  =  \lambda x_n(1 - x_n), $$

\noindent where $ x_{n} \in [0,1]$ can be interpreted as a measure of the population in the $n$th time interval, and $\lambda \geq 0$ represents the intrinsic growth rate.

In this paper, we discuss the logistic map under two different conditions, when $\lambda $ is fixed and when it is random. In order to write out the map under the two conditions some notation is needed.
Define,
$$ \vec{\lambda} =  (\lambda_1,  \lambda_2, \ldots  \, \, \, );  \lambda_i \stackrel{iid}{\sim} {\mathcal U}[a, b],$$

\noindent
where ${\mathcal U}[a, b]$ indicates the uniform distribution with support on the interval $[a,b]$, and $ \sigma (\vec{\lambda})$ to be the shift operator such that $\sigma : (\lambda_1, \lambda_2 , \ldots ) \mapsto  (\lambda_2, \lambda_3,\ldots)$.  Let $\pi(\vec{\lambda}) = \lambda_1$ be the projection onto the first coordinate of $\vec{\lambda}$.  Then

\begin{equation} S_{\lambda}(x) = \lambda x (1 - x) \, \, 
\end{equation}
denotes the logistic map when $\lambda $ is fixed, and

\begin{equation}  \label{Tmapeq}
 T(\vec{\lambda},x) =  (\sigma(\vec{\lambda}), \pi(\vec{\lambda}) x (1 - x) ) \, \,
\end{equation}
denotes the stochastic version of the logistic map as a skew product.  Note that $S_{\lambda_0}(x) = T(\vec{\lambda},x)$ with $\vec{\lambda} = (\lambda_0, \lambda_0, \lambda_0, \dots) $.
 An example of the output of the stochastic logistic map is shown in Figure~\ref{Tmapfig}, where $\lambda_i \stackrel{iid}{\sim} {\mathcal U}[2, 2.5].$

\begin{figure} [H]
\begin{center}
\includegraphics[scale=.5]{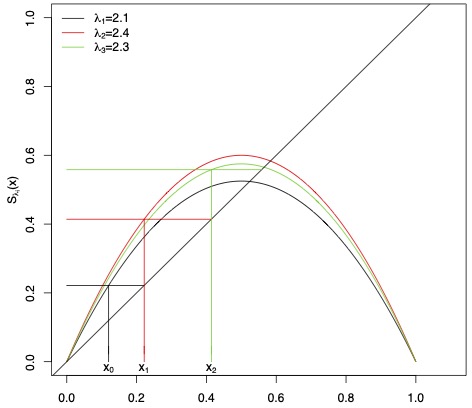}
\caption{$T^i(\vec{\lambda},x_0)$ for $i = 1, 2, 3$  where $\vec{\lambda}= (2.1, 2.4, 2.3, \dots) $, and $x_0 = 0.12$. The graphs of $S_{2.1}(x)$, $S_{2.4}(x)$ and $S_{2.3}(x)$ are shown,  with the line $y=x$ superimposed.}
\label{Tmapfig}
\end{center}
\end{figure}

\noindent We denote iterates of $T(\vec{\lambda},x)$ by $T^i(\vec{\lambda},x),$ where 
$$T^2(\vec{\lambda},x) = T(T(\vec{\lambda},x)),\quad \dots \quad T^i(\vec{\lambda},x) = T(T^{i-1}(\vec{\lambda},x)) .$$
Define $\Pi(\vec{\lambda}, x)  = x$ as the projection of the pair $(\vec{\lambda}, x)$ onto the state space, $[0,1]$.
We call the sequence $x_0, x_1 = \Pi\left({T(\vec{\lambda}, x_0)}\right),  \dots, x_n = \Pi\left({T^n(\vec{\lambda},x_0)}\right), \dots $ the {\it orbit of $x_0$} under $T(\vec{\lambda},x)$.  If the initial value is a random variable independent of $\lambda$, then the set of orbits  form a Markov Chain with state space some subset of $[0,1]$, depending on the distribution of $\lambda$.  We will use capital letters to denote this Markov Chain: $X_1 = \Pi\left({ T(\vec{\lambda},X_0) }\right), \dots
X_{i+1} = \Pi\left({ T(\sigma^i(\vec{\lambda}),X_i) }\right) $.

Note that in order to keep $x$ in the interval $ [0,1]$ for all iterations (to keep the unit interval invariant) $ \lambda $ must lie in the interval $ [0,4]$.

\subsection{The Deterministic Logistic Map} \label{sec:deterministicbif}

The deterministic logistic map, $S_{\lambda}(x)$, has been well studied in the literature \cite{May, Devaney, Chaos}.  For our purposes, it is important to know that for $\lambda$ in the interval $(3, 1 + \sqrt{6})$, the two fixed points (at zero and $\frac{\lambda - 1}{\lambda}$) are unstable, and the logistic map has a stable period-2 cycle.  For these values of $\lambda$, almost all initial values in (0,1) have orbits that converge to the period two orbit \cite[p.~359]{Chaos}:

\begin{equation} p(\lambda) = \frac{(\lambda + 1) - \sqrt{(\lambda - 3)(\lambda + 1)}}{2 \lambda} \label{peq} \end{equation}  \begin{equation} q(\lambda) = \frac{(\lambda + 1) + \sqrt{(\lambda - 3)(\lambda + 1)}}{2 \lambda}. \label{qeq} \end{equation} 
As the parameter $\lambda$ increases beyond $1 + \sqrt{6}$, the map undergoes a sequence of period-doublings. More details about the deterministic logistic map, including a bifucation diagram and additional references, are given in Appendix \ref{sec:deterministic_background}.

\subsection{Bifurcation Diagram for the Stochastic Logistic Map}

We can numerically construct a bifurcation diagram for the stochastic logistic map by plotting the long-term behavior of a sample path as a function of the mean of a distribution of $\lambda$-values.  Figure~\ref{BifurcT} shows an example of such a diagram where the parameter is uniformly distributed in a small interval:  $\lambda \sim {\mathcal U}[\bar{\lambda} - \Delta \lambda, \bar{\lambda} + \Delta \lambda]$, $\Delta \lambda = 0.1$, and $\bar{\lambda}$ ranges from $0 + \Delta \lambda = 0.1$ to $4-\Delta \lambda = 3.9$.  For $\bar{\lambda} <  1 - \Delta \lambda = 0.9$, $S_\lambda(x)$ has a unique attracting fixed point at $x  = 0$ for all $\lambda \in (\bar{\lambda}-\Delta \lambda, \bar{\lambda} + \Delta \lambda)$.  Therefore, all orbits under $T(\vec{\lambda}, x)$  converge to 0, and the bifurcation diagram in Figure~\ref{BifurcT} shows this with a horizontal line at 0 for $\bar{\lambda} \in [0+\Delta \lambda, 1 - \Delta \lambda]$.  For $\bar{\lambda} \in (1 + \Delta \lambda = 1.1, 3 - \Delta \lambda = 2.9)$, $S_{\lambda}(x)$ has a non-zero attracting fixed point at $\displaystyle{\frac{\lambda -1}{\lambda}}$ for all $\lambda \in (\bar{\lambda}-\Delta \lambda, \bar{\lambda} + \Delta \lambda)$.  For these values of $\bar{\lambda}$,  orbits of $T(\vec{\lambda},x)$ eventually enter the interval
$\displaystyle{\left({ \frac{\bar{\lambda} - \Delta \lambda - 1}{\bar{\lambda} - \Delta \lambda},\frac{\bar{\lambda} + \Delta \lambda -1}{\bar{\lambda} + \Delta \lambda} }\right)}$ with probability 1 and, once in this interval, they never leave it.  Figure~\ref{BifurcT} shows a single curved strip of width $\frac{2 \cdot \Delta \lambda}{(\bar{\lambda} + \Delta \lambda)(\bar{\lambda} - \Delta \lambda)}$ centered at $\displaystyle{\frac{\bar{\lambda}^2 - \bar{\lambda} - (\Delta \lambda)^2}{\bar{\lambda}^2 - (\Delta \lambda)^2} }$.  Figure~\ref{BifurcT} also shows a distinct period 2 regime when $\bar{\lambda} \in (3 + \Delta \lambda, 1 + \sqrt{6} - \Delta \lambda)$.

 Bifurcations of families of stochastic logistic maps have also been explored for other types of distributions of the parameter values.  In \cite{Schenk-Hoppe1997}, for example, the ``dichotomous" case, where the distribution is supported on two $\lambda$-values, is studied.    Implicit in the exploration of a stochastic system via these bifurcation diagrams (and, indeed, in the presentation of most numerical simulations of random phenomena) is that they represent generic behavior, i.e., that almost all orbits have the same limiting distribution: we would get roughly the same bifurcation diagram regardless of which value we chose for $x_0$.   This characterizes {\it ergodicity}, which we define in the next section.

In the rest of this paper we are interested in statistical characterizations of orbits when $\lambda$-values are supported on small intervals that are completely contained in the stable periodic regimes of the deterministic logistic map.  In particular, we consider the average behavior along orbits of the map $T(\vec{\lambda}, x)$, where $\vec{\lambda}$ is an i.i.d. sequence of random variables $\lambda_1, \lambda_2, \dots $  with common distribution $g(\lambda)$.  In order to make statements about ``the" average behavior, we need to establish the existence of a well-defined measure for the distribution of $x$-values after many iterations.  This is the content of the next section.

\begin{figure} [h!]
\begin{center}
\includegraphics[width = .55\textwidth]{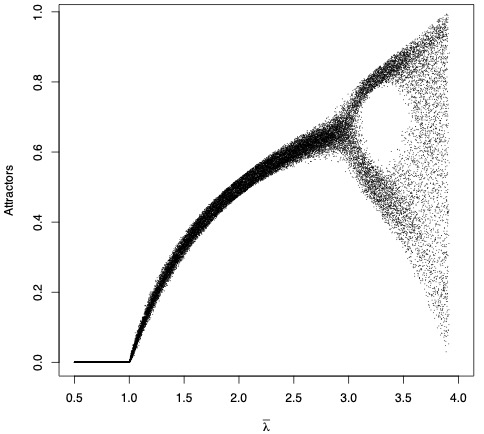}
\caption{ A bifurcation diagram for $T(\vec{\lambda}, x)$ with  $\Delta \lambda = 0.1$ and  $\lambda_i \stackrel{\ i.i.d.}{\sim} {\mathcal U}[\bar{\lambda} - \Delta \lambda,\bar{\lambda} + \Delta \lambda]$, $\bar{\lambda} \in [0 + \Delta \lambda, 4 - \Delta \lambda]$.  For each value of $\bar{\lambda}$, plotted on the horizontal axis, we begin with 100 values of $x_0$ chosen from a uniform distribution on $[0,1]$, and plot their 1000$^{th}$ iterate under $T(\vec{\lambda}, x)$ on the vertical axis, where $\vec{\lambda}$ is an i.i.d. sequence that is different for each simulation.  In other words, we simulated 100 sample paths of the Markov Chain defined by the stochastic logistic map for each starting value, $x_0$.  The periodic structure from the deterministic setting is still visible  for $\bar{\lambda}$-values in the period-2 regime, but the bifurcations are blurred.   }
\label{BifurcT}
\end{center}
\end{figure}

\subsection{Ergodicity and Stability}  \label{sec:ergodicity}
We now consider the stochastic logistic map as a map on distributions.  This requires some new definitions.

 Let ${\mathcal B}$ denote the Borel sigma-algebra on the state space $\Omega = [0,1]$ and ${\mathcal M}$ the space of probability measures defined on $(\Omega, {\mathcal B})$.
The stochastic map, $T(\vec{\lambda}, x)$ induces a map on ${\mathcal M}$ as follows.  Assume that $\vec{\lambda}$ is a sequence of i.i.d.\ random variables with common density $g(\lambda)$ supported on $[a,b] \subset [0, 4]$.  For $\mu \in {\mathcal M}$ define the operator on measures:  
\begin{equation} \label{eq:PF}
P^*\mu (A) = \int_0^4 \mu\left({S_{\lambda}^{-1}(A)}\right) g(\lambda) \, d \lambda 
\end{equation}

The operator $P^*$ is called the
{\it Foias operator} (\cite{Lasota}, Def. 12.4.2), and coincides with the {\it Perron-Frobenius operator}
when $T$ is an arbitrary map on the state space.

Fixed points of this map take the place of fixed points in the deterministic setting.

\begin{defn} An { \bf invariant measure} is a measure $\mu^* \in {\mathcal M}$,  such that, $P^* \mu^* (A) = \mu^*(A)$ for all $A \in {\mathcal B}$.  \end{defn}
Thus, an invariant measure is a fixed point of the operator, $P^*$.  
  We can also define  stability for stochastic dynamical systems via the operator $P^*$:
\begin{defn} Suppose $\mu^*$ is the unique invariant measure for the operator $P^*$.  Then the system $T(\vec{\lambda},x)$ is {\bf stable in distribution} if, for any $\mu \in {\mathcal M}$
$$ \left\vert{ \lim_{n \rightarrow \infty} (P^*)^n \mu (A)  - \mu^*(A) }\right\vert = 0 , \quad \forall \, A \in {\mathcal B} $$
where $P^*$ is the Foias operator associated with $T$.
\end{defn}

A set $A \in {\mathcal B}$ is invariant under the system  $ T(\vec{\lambda}, \cdot)$ if  $T^{-1}(\vec{\lambda}, A) = A$ 
for almost every $\vec{\lambda}.$

A probability measure $\mu$ is {\it ergodic} if $\mu(A) = 0$ or $1$ for every invariant set $A$.  Suppose $\mu^*$ is invariant and ergodic.  Then $\mu^*$ is the unique invariant probability, and the ergodic theorem tells us, that for any $A \in {\mathcal B}$, for almost every $x_0$, with respect to $\mu^*$:
\begin{equation} \label{eq:ergodic}
\lim_{n \rightarrow \infty} \frac{1}{n} \sum_{i=0}^{n-1} \chi_A(x_i) = \mu^*(A), 
\end{equation}
where $\chi_A(x)$ is the indicator function of the set $A$ and $\{ x_i\}_{i=0}^\infty$ is an orbit of $x_0$ under the stochastic map.  If, in addition, the invariant measure $\mu^*$ is stable, then 
$$ \mu^*(A) =  \lim_{n \rightarrow \infty} (P^*)^n\mu(A), \quad \forall \mu \in {\mathcal M}.$$
Thus, we can estimate the long-term behavior along a $\mu^*$-generic orbit by iterating simulated estimates of arbitrary distributions.  In fact, this is the justification behind the algorithm used to generate the stochastic bifurcation diagram shown in Figure \ref{BifurcT}.

The following result from \cite[p.~343]{Bhattacharya2007} shows that the stochastic logistic map with continuous parameter distributions supported in a given periodic regime is stable. 

Assume that $S_{\lambda}(x)$ has a stable periodic orbit with period $m$ for all $\lambda$ in the support of $g(\lambda)$.

\begin{prop} \label{thm:stable}
\cite{Bhattacharya2007} If
$$ E[\log \lambda] > 0 \text{ and } E[\left\vert{\log(4 - \lambda)}\right\vert] < \infty $$
then
\begin{enumerate}
\item the stochastic dynamical system $T(\vec{\lambda},x)$ has a unique invariant probability  measure $\mu^* \in {\mathcal M}$, and
\item $\displaystyle{\left\vert{ \lim_{n \rightarrow \infty}(P^*)^{mn}\mu(A) - \mu^*(A)}\right\vert = 0 }$ for all $A \in {\mathcal B}, \mu \in {\mathcal M}$, and $m$ is the period associated with the stable periodic orbit.
\end{enumerate}
\end{prop}

In this paper, we consider the situation when $g(\lambda)$ is the uniform density on an interval $[a,b]$, where $[a,b]$ is entirely within either the non-zero fixed point regime: (1,3), the period-2 regime: $(3,1+\sqrt{6})$, or the period-4 regime: $(1+\sqrt{6},3.54409)$.  In all cases, $a$ and $b$ are bounded away from 1 and 4,  so that $E[\log \lambda]> 0$ and
$E[\left\vert({ \log(4 - \lambda)}\right\vert]  < \infty$.  Therefore, Proposition \ref{thm:stable} applies and we know that the long-term distribution of almost all orbits is well-defined.  Thus, ``the expected value"  exists in the various periodic regimes, and  we are in a position to compare the means of the deterministic and stochastic logistic maps in these regimes.  Note that an explicit formula for the unique invariant measure is not known, but the existence and stability of the invariant measure ensure that numerical simulations should reflect the generic behavior of sample paths of the system.


We can verify the theoretical results numerically by exploring more closely ``slices" of the stochastic bifurcation diagram shown in Figure~\ref{BifurcT}.

The red-solid densities in Figure~\ref{ItPerOne} represent three different slices (panel (a), (b), and (c), respectively) of the bifurcation diagram of $T(\vec{\lambda},x)$ shown in Figure~\ref{BifurcT} using $\Delta \lambda = 0.024$,  with $\lambda$-values supported in different periodic regions.  Plot (a) is a vertical slice of a bifurcation diagram of $T(\vec{\lambda},x)$ with $ \lambda \sim {\mathcal U}[1.508 - \Delta \lambda, 1.508 + \Delta \lambda]$, within the non-zero fixed point regime.  The figure shows the empirical distribution of the $x$ values for different iterations, i.e., it gives a numerical approximation of iterates of the Foias operator, $P^*$ given in Equation \eqref{eq:PF}, applied to the uniform distribution.  To generate the numerical approximation, we start with $1000$ uniformly drawn random values of $X_0$.  This gives a numerical approximation to $\mu$,  the uniform distribution on $[0,1]$ indicated by the purple-two-dash curve in Figure~\ref{ItPerOne}. The blue-long-dash curve in panel (a) represents the empirical distribution of $X_1 = T(\vec{\lambda}, X_0)$ with $\lambda \sim {\mathcal U}[1.508 - \Delta \lambda, 1.508 + \Delta \lambda]$.  The other curves give the empirical distributions of $X_{10}, X_{50}, X_{100}$ and finally $X_{1000}$, in red-solid.  Plots (b) and (c) show the evolution of the uniform distribution under $P^*$ when the parameter values are in the period-2 and period-4 regions; the distributions were created using $\lambda \sim {\mathcal U}[3.208 - \Delta \lambda, 3.208 + \Delta \lambda]$ and $\lambda \sim {\mathcal U}[3.508 - \Delta \lambda, 3.508 + \Delta \lambda]$, respectively.  The periodicity of the deterministic systems is reflected in the multi-modality of the limiting distributions, although we note that these empirically derived distributions suggest asymmetries.
The numerical results demonstrate that the 50th iteration of $P^*$  gives a fairly close approximation to the stable invariant distribution, $\mu^*$.  The choice of $\Delta \lambda = 0.024$ was made so that the visual intuition becomes clear, but the theoretical results do not depend on any specific value of $\Delta \lambda$, as long as the support of $\lambda$ stays in a particular periodic regime. (See Section \ref{sec:deterministicbif}.)


\begin{figure} [H]
\begin{center}
\includegraphics[width=.48\textwidth]{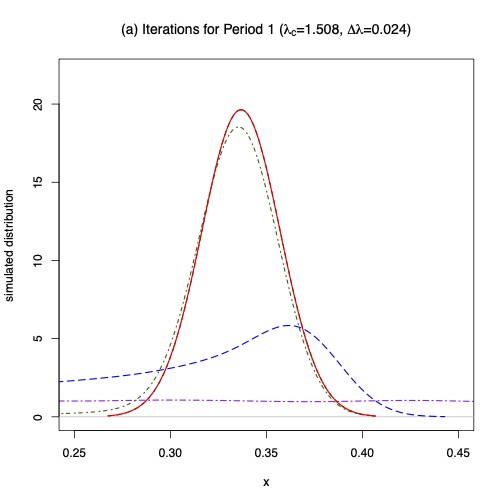} \includegraphics[width=.48\textwidth]{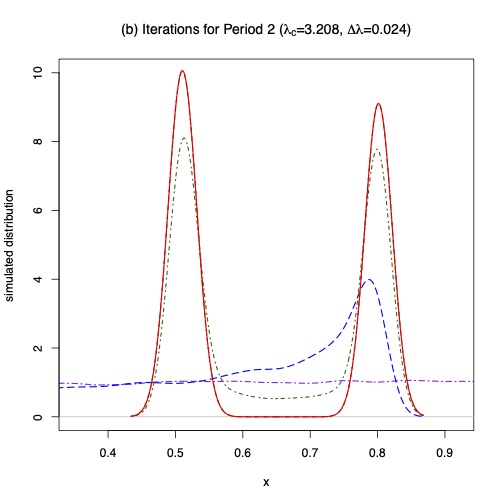} \includegraphics[width=.48\textwidth]{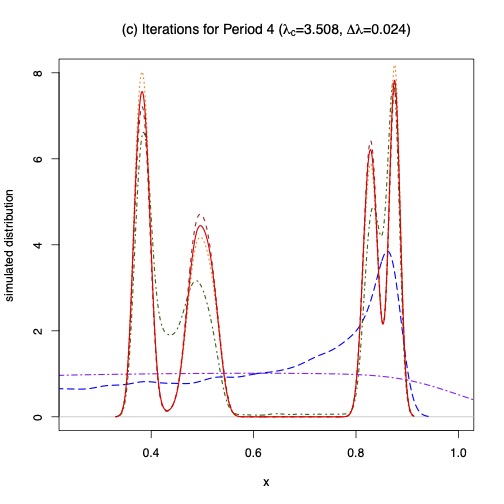} \includegraphics[width=0.48\textwidth]{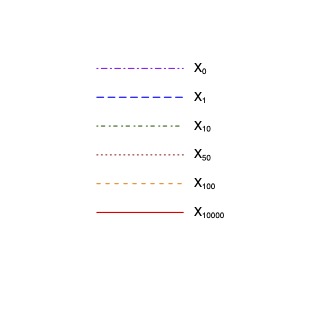}
\caption{ The distribution of the state variable with a uniform initial distribution. The initial, $1$st, $10$th, $50$th, $100$th, and $10000$th iterations are shown in purple-two-dash, blue-long-dash, green-dot-dash,  brown-dot-dot, orange-dash, and red-solid respectively.
The parameter values  in each of the three plots are distributed according to  (a) $\lambda \sim {\mathcal U}[1.508 - \Delta \lambda, 1.508 + \Delta \lambda]$, (b) $\lambda \sim {\mathcal U}[3.208 - \Delta \lambda, 3.208 + \Delta \lambda]$, and  (c) $\lambda \sim {\mathcal U}[3.508 - \Delta \lambda, 3.508 + \Delta \lambda]$.  Note that the red-solid densities in each plot give an empirical estimate of the invariant distribution for the given distribution on $\lambda$.  }
\label{ItPerOne}
\end{center}
\end{figure}

\section{Stochastically-induced changes in the average behavior.}

\label{sec:simulations}

\subsection{Mean versus Mean}

If $g(\lambda)$ has support in one of the periodic regimes, then Proposition \ref{thm:stable} gives (let $A = (0,1)$):

$$  \lim_{n \rightarrow \infty} \frac{1}{n} \sum_{i=0}^{n-1} T^i(\vec{\lambda}, x) = \int_0^1 x\, d\mu^*(x) = E_{\mu^*}[X], $$

for almost all initial values of $x.$

The result in this paper is a comparison of the expected value of $X$ under $T(\vec{\lambda}, \cdot)$ with the mean along the periodic orbit under $S_\lambda(x) = \lambda x(1-x)$, the deterministic map, using the expected value of $\lambda $. 
Similar investigations have been discussed in \cite{Haskell2005}  in the context of the Beverton-Holt equation, and in \cite{Hinow2016}, in the context of self-limiting exponential growth.   In the next section we compare the average behavior of the stochastic logistic map with the deterministic map in several parameter regimes, uncovering a pattern of alternating larger and smaller expected values under stochasticity.

\subsection{Period One \label{subsecP1}}

Suppose $\lambda \sim {\mathcal U}[a,b]$.  In Figure~\ref{HistOne} we compare the ``deterministic mean", i.e., the long-term average of orbits under $S_{\bar{\lambda}}$ where $\bar{\lambda} = \displaystyle{\frac{a + b}{2}}$, with the long-term behavior under the stochastic dynamical system, $T(\vec{\lambda}, \cdot)$.  The left panel of Figure~\ref{HistOne}  shows the stable fixed point at $ \bar{\lambda} = 1.508$ (red line), the limit set for almost all orbits under $S_{\bar{\lambda}}$.
In the same panel we plot the empirical distribution of $\left({ P^*}\right)^{10000}(\mu)$ (black histogram), where $\mu$ is the uniform distribution on $(0,1)$,  $\Delta \lambda = 0.024$, and $\lambda \sim {\mathcal U}[1.508 - \Delta \lambda, 1.508 + \Delta \lambda]$; the average of the empirical distribution is given by the blue line.  Because the system is stable, we can interpret $\left({ P^*}\right)^{10000}(\mu)$ as an estimate of the invariant distribution, $\mu^*$.  The right panel of Figure~\ref{HistOne} shows a magnification of the left panel.   Observe that the blue line is to the left of the red line, hence the (empirically estimated) average of  $\mu^*$ is less than the stable fixed point of the deterministic map when $ \bar{\lambda}= 1.508 $.   Since the stable fixed point is the average along almost every orbit of the deterministic map, Figure~\ref{HistOne} gives a visual comparison of the mean of the stochastic system, where the parameter values have mean $\bar{\lambda}$ (blue line),  with the mean of the deterministic system with parameter value  equal to $\bar{\lambda}$ (red line).

\begin{figure} [H]
\begin{center}
\includegraphics[width = .47\textwidth]{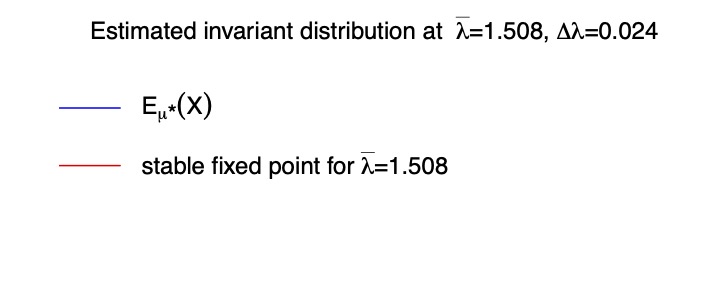} \quad
\includegraphics[width = .47\textwidth]{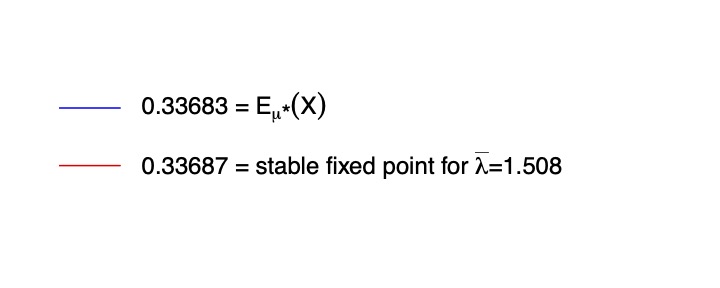}\\
\includegraphics[width = .47\textwidth]{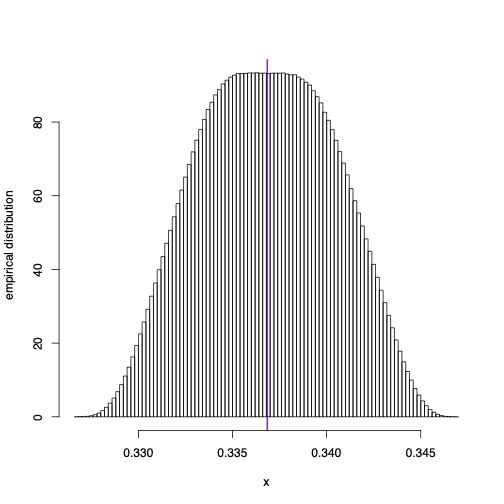} \quad
\includegraphics[width = .47\textwidth]{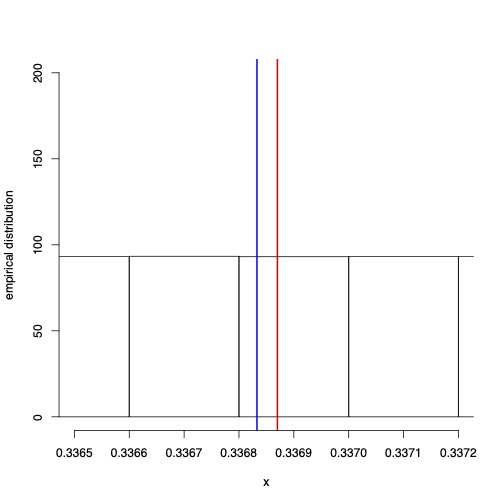}
\caption{An illustration of the result proved in Theorem \ref{PeriodOneTheorem}. The left panel shows an estimated histogram of the $10000^{th}$ iterate of $P^*$, the Foias operator associated with $T(\vec{\lambda},x)$.  In this simulation,  the parameter value at each iteration is given by $ \lambda \sim {\mathcal U}[1.508 - \Delta \lambda, 1.508 + \Delta \lambda]$ with $\Delta \lambda = 0.024$; 20,000 initial $x$-values were drawn from a uniform distribution on $[0,1]$.  Since the system has a unique stable invariant measure, the histogram is an estimate of the distribution of the limiting invariant measure, $\mu^*$.  The right panel shows a magnification of the central portion of the histogram, highlighting that the mean of the empirical histogram (blue line) is less than the stable fixed point of the logistic map with the average parameter value, $\bar{\lambda} = 1.508$ (red line).  }
\label{HistOne}
\end{center}
\end{figure}


The following theorem
is a theoretical validation of the simulation shown in Figure \ref{HistOne}. A similar result is proved in \cite{Haskell2005}, where it is noted that the proof, an application of Jensen's inequality, only depends on the convexity of the map.  This result shows that the ``stochastic mean" will always be less than the deterministic mean. In other words, the presence of randomness lowers the long-term average in the period-one regime.
While the proof of Theorem \ref{PeriodOneTheorem} is straightforward, the fact that the relationship between deterministic and stochastic means ``flips" in the period 2 regime is not intuitive.  This new result is proved in Theorem \ref{PeriodTwoTheorem}.

Define $x^*(\lambda) = \displaystyle{\frac{\lambda -1}{\lambda}}$ to be the non-zero fixed point of the map $S_{\lambda}(x)$.
Recall that this fixed point is {\it attracting} for $\lambda \in (1, 3)$ and, therefore, the C\'{e}saro average along almost any orbit converges to $x^*(\lambda)$:
$$ \lim_{n \rightarrow \infty} \frac{1}{n} \sum_{i = 0}^{n-1} S^{i}_\lambda(x) = \frac{\lambda -1}{\lambda} = x^*(\lambda)
\quad \hbox{if } \; \lambda \in (1,3). $$

\begin{thm} Let $g(\lambda)$ be a continuous density whose support is in the period one region, $(1,3) $.  Then the mean of $X$ under the stochastic logistic map $T(\vec{\lambda}, \cdot)$  is {\bf less} than the stable fixed point of the deterministic map 
$S_{\bar{\lambda}}(x)$, where  $\bar{\lambda} = \int_0^4 \lambda g(\lambda) d \lambda$. 

\label{PeriodOneTheorem}
\end{thm}

\begin{proof}
Let $ \mu^* $ be the unique invariant distribution of the Foias operator associated with $T(\vec{\lambda}, \cdot)$. Since $\lambda$ is independent of $X$:
\begin{eqnarray*}
E_{\mu^*}[X] &=& \int_0^1 x d\mu^*(x) \\
&=& \int_0^1 x dP^*(\mu^*)(x) \\
&=& 
\int_0^1 x \int_0^4 d\mu^*\left({S_\lambda^{-1}(x)}\right) g(\lambda) d \lambda
\quad \quad \quad \hbox{by the definition of } P^* \\
&=& \int_0^1 y(1-y) \int_0^4 \lambda  g(\lambda) d\lambda \, \, d\mu^*(y)  \quad \quad \quad 
\hbox{letting } x = S_{\lambda}(y) = \lambda y(1-y)\\
&=& 
\bar{\lambda} \int_0^1 (y  - y^2)  d\mu^*(y)
\quad \quad \quad \quad \quad \quad \quad \quad
\hbox{since } \bar{\lambda} = \int_0^4 \lambda g(\lambda) d \lambda \\
&=& \bar{\lambda} \left({  E_{\mu^*}[X] - E_{\mu^*}[X^2] }\right) \\
&\le& \bar{\lambda} \left({ E_{\mu^*}[X] - E_{\mu^*}[X]^2 }\right)  \quad \hbox{since, by Jensen's inequality }
E_{\mu^*}[X]^2 \le E_{\mu^*}[X^2]
\end{eqnarray*}
\begin{eqnarray*}
&\Rightarrow  (\bar{\lambda} - 1) E_{\mu^*}[X] - \bar{\lambda} E_{\mu^*}^2[X] \ge 0 \\
&\Rightarrow  E_{\mu^*}[X]\left({\bar{\lambda} -1 - \bar{\lambda} E_{\mu^*}[X] }\right) \ge 0 \\
&\Rightarrow E_{\mu^*}[X] \le \frac{\bar{\lambda} -1}{\bar{\lambda}} = x^*(\bar{\lambda}) 
\end{eqnarray*}
since $E_{\mu^*}[X] \ge 0 $.  Furthermore, the inequality is strict unless $E_{\mu^*}[X^2] = 0 $.
\end{proof}

Note: the final inequality in the proof of Theorem~\ref{PeriodOneTheorem} holds for $\lambda $ in any region. However, 
$\displaystyle{\frac{\bar{\lambda} - 1}{\bar{\lambda}} }$ is the average of the deterministic system only in the non-zero stable period-one regime, $\lambda \in (1,3)$.   Therefore, the proof of theorem~\ref{PeriodOneTheorem} does not tell us anything about the relationship between the stochastic and deterministic means when $g(\lambda)$ is supported in other regions.

\subsection{Period Two} \label{sec:period2}

Figure~\ref{HistTwo} shows a histogram of $\left({P^*}\right)^{10000}(\mu) \approx \mu^*$ where the parameter values have a uniform distribution:  $\lambda \sim {\mathcal U}[3.208 - \Delta \lambda, 3.208 + \Delta \lambda]$, $\Delta \lambda = 0.024$.  As in Figure~\ref{HistOne},   the empirical stochastic and deterministic means with $\bar{\lambda} = 3.208$ are indicated by green and orange vertical lines, respectively.  The right panel of Figure~\ref{HistTwo} shows a magnified region of the left panel, showing that the orange line is to the left of the green line, so the   empirical estimate of the stochastic mean  is {\it greater} than the average along the entire stable period-2 orbit of the average parameter value. Note that this is opposite of the situation depicted in Figure~\ref{HistOne}.  The numerical results, along with other simulated histograms  of $(P^*)^{10000}(\mu)$ (not shown) when  $ \lambda $ is in the period two region, demonstrate that the expected value of $\pi(T(\vec{\lambda},X))$ is {\bf greater} than  the mean along the stable period-2 orbit of the deterministic map evaluated at $\bar{\lambda}$.  In other words, the relationship between the stochastic and deterministic means in the period two region is the reverse of the relationship in the period one region.  We prove that this is the case in Theorem \ref{Period2Theorem}.

We can gain some intuition about this reversal from the following observation, also depicted in Figure~\ref{HistTwo}.  
For $\lambda \in [3.208 - \Delta \lambda, 3.208 + \Delta \lambda]$, the map $S_{\lambda}$ has a stable period 2 orbit:
$\{ p(\lambda), q(\lambda) \}$ where $p(\lambda) <  q(\lambda)$ are given by Equations \eqref{peq} and \eqref{qeq}, respectively. 
Thus, almost all orbits of the stochastic system will eventually enter  two disjoint intervals, $I_p$ and $I_q$, where $I_q$ is the image of the interval: $ [p(3.208 + \Delta \lambda),p(3.208 - \Delta \lambda)]$ and $I_p$ is the image of the interval $[q(3.208 - \Delta \lambda),q(3.208 + \Delta \lambda)]$.  A generic orbit will alternately visit $I_p$ and $I_q$ and thus the invariant measure, $\mu^*$, can be written as a sum of two measures, $\mu_p^*$ and $\mu_q^*$, supported on $I_p$ and $I_q$, respectively, and $E_{\mu^*}[X]$ will be the average of $E_{\mu_p^*}[X]$ and $E_{\mu_q^*}[X]$.  Figure \ref{HistTwo} compares $E_{\mu_p^*}[X]$ and $E_{\mu_q^*}[X]$ with the stable period-2 points of the deterministic system with the average parameter value, $\bar{\lambda} = 3.208$.  We see that $E_{\mu_p^*}[X] > p(\bar{\lambda})$ and $E_{\mu_q^*}[X] < q(\bar{\lambda})$, but we formally prove in the Appendix that 
$$\left[{ E_{\mu_p^*}[X] - p(\bar{\lambda}) }\right]
>  \left[{ q(\bar{\lambda}) - E_{\mu_q^*}[X] }\right] , $$
so that $E_{\mu^*}[X] > \frac{p(\bar{\lambda}) + q(\bar{\lambda})}{2} $.  This is the conclusion of Theorem \ref{Period2Theorem}.

\begin{figure} [H]
\begin{center}
\includegraphics[width = .47\textwidth]{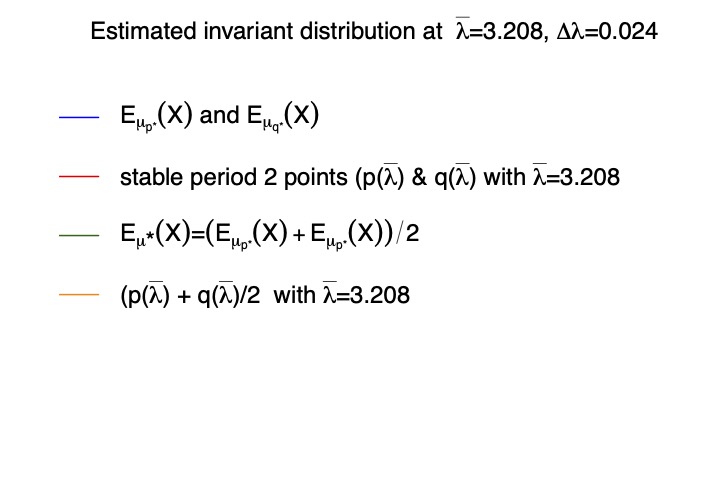} \quad
\includegraphics[width = .47\textwidth]{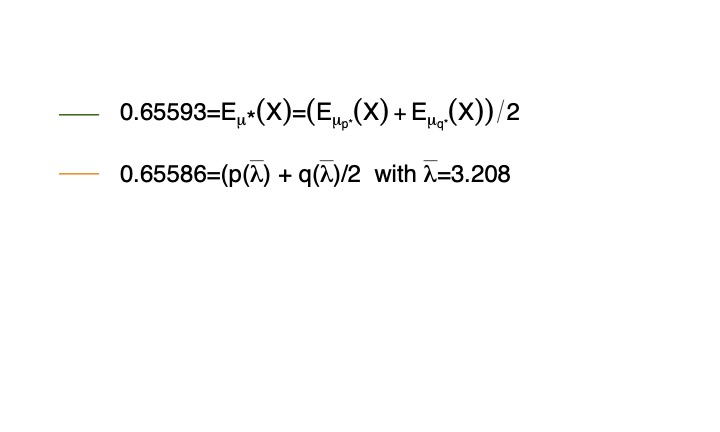}\\
\includegraphics[width = .47\textwidth]{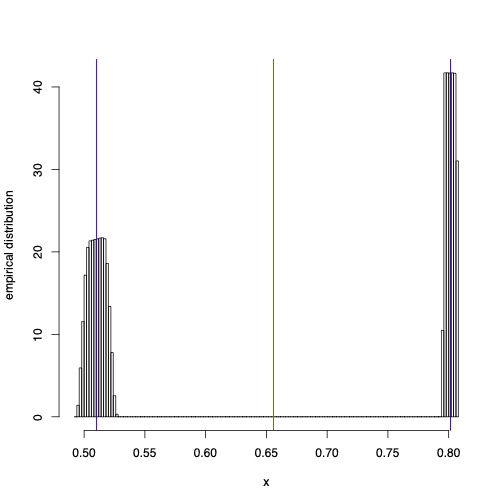}
\quad
\includegraphics[width = .47\textwidth]{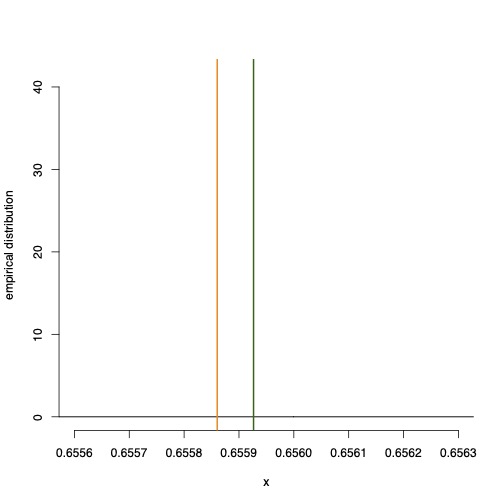}
\caption{ Numerically estimated histogram of the $10000^{th}$ iterate of $P^*$, the Foias operator associated with $T(\vec{\lambda},x)$.  In this simulation,  the parameter value at each iteration is given by $ \lambda \sim {\mathcal U}[3.208 - \Delta \lambda, 3.208 + \Delta \lambda]$ with $\Delta \lambda = 0.024$; 20,000 initial $x$-values were drawn from a uniform distribution on $[0,1]$.   Since the system has a unique stable invariant measure, this histogram is an estimate of the distribution of the limiting invariant measure, $\mu^*$.  The right panel shows a magnification of the central portion of the bimodal histogram, highlighting that the mean of the empirical histogram (green line, $(E_{\mu_p^*}[X] + E_{\mu_q^*}[X])/2$) is greater than the average of the stable fixed points of the invariant distribution at the average parameter value, $\bar{\lambda} = 3.208$ (orange line, $(p(\bar{\lambda}) + q(\bar{\lambda}))/2$). }
\label{HistTwo}
\end{center}
\end{figure}

The previous motivation leads to the following theorem.  Rather than a convexity argument similar to the one used to prove Theorem \ref{PeriodOneTheorem}, the proof presented here relies on several rather technical lemmas which are proved in the Appendix. 
We also restrict ourselves to the case where $g(\lambda)$ is  the uniform density on an interval.

Suppose $\lambda \sim {\mathcal U}[\bar{\lambda} - \Delta \lambda, \bar{\lambda} + \Delta \lambda]$ where $(\bar{\lambda} - \Delta \lambda, \bar{\lambda} + \Delta \lambda) \subset  (3, 1 + \sqrt{6})$.  Note that this means that $S_{\lambda}$ has a stable period-2 orbit for all values of $\lambda$.  We also require $\Delta \lambda$ to be small enough so that the invariant measure is supported on two disjoint intervals.  See below and Lemma \ref{LambdaClaim} for details.

\begin{thm}
\label{Period2Theorem}
 In the period two region, with $\lambda$ uniformly distributed as described above, the expected value of $X$ under $\mu^*$, the stochastic logistic map (i.e., $\Delta \lambda > 0$), is {\bf greater} than  the average of the stable period-2 points of the deterministic map with the average parameter value,  $\bar{\lambda}$. \label{PeriodTwoTheorem}
 $$E_{\mu^*}[X] > \frac{p(\overline{\lambda})+ q(\overline{\lambda})}{2} = \frac{\overline{\lambda} + 1}{2 \overline{\lambda}}.$$

\end{thm}

Suppose  $\Delta \lambda$ is small enough so that the support of the invariant distribution, $\mu^*$, consists of two disjoint intervals (see Lemma \ref{LambdaClaim} in the Appendix).  Then define the measures $\mu_p^*$ and $\mu_q^*$ as follows:
\begin{equation} \label{inv_measures}  
\mu_p^*(A) = 2 \mu^*\Bigg(A \cap \bigg[0, \frac{\bar{\lambda}-1}{\bar{\lambda}} \bigg] \Bigg) \quad \quad
\mu_q^*(A) = 2 \mu^*\Bigg(A \cap \bigg(\frac{\bar{\lambda}-1}{\bar{\lambda}}, 1 \bigg] \Bigg) 
\end{equation}
Thus, $\mu_p^*$ gives the left peak of the invariant distribution, and $\mu_q^*$ gives the right peak.  

 The proof of Theorem \ref{Period2Theorem} involves showing  that, for small $\Delta \lambda > 0$,  the expected value of the lower peak, $E_{\mu_p^*}[X]$, is farther from $p(\bar{\lambda})$ than the distance of the expected value of the upper peak, $E_{\mu_q^*}[X]$, to $q(\bar{\lambda})$.  The entire proof is in the Appendix.

\subsection{Period Four}

Figure~\ref{HistFour} shows the average of the four stable fixed points for $ \bar{\lambda} = 3.508 $, corresponding to the four points on the vertical slice of the bifurcation diagram (see Figure~\ref{bifurc}) of $S_\lambda(x)$ with $ \lambda = 3.508$ (orange line).  In the same figure, we plot the average of the points that make up the empirical histogram (green line). Observe in the right panel of  Figure~\ref{HistFour}, the zoomed in version of the left panel, the  $x$-coordinate of the green line is less than that of the orange line, so the average of the empirical estimation of $\mu^*$ is strictly less than the average of the stable period-four points, just as in the period one region. Figure~\ref{HistFour}, and other histograms of the last four iterations (not shown) of $T(\vec{\lambda},x)$ for $ \lambda $ in the period four region, demonstrate that the expected value of $\mu^*$ under the stochastic logistic map is {\bf less} than or equal to the long-term average under the corresponding deterministic map.  Thus the behavior in the period four region is the same as in the period one region (and has flipped from the behavior in the period two region).

\begin{figure} [H]
\begin{center}
\includegraphics[width = .47\textwidth]{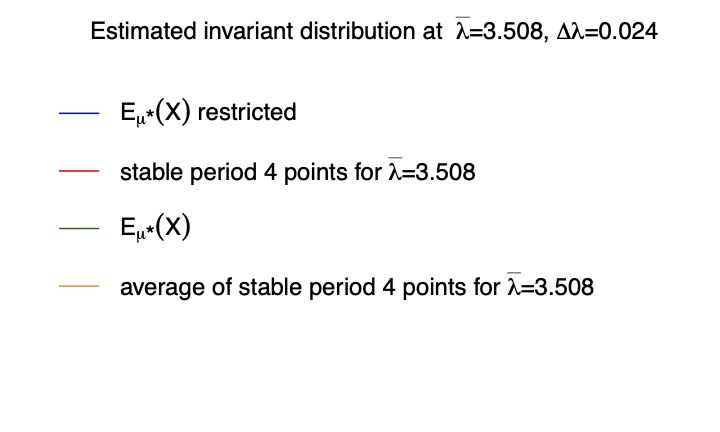} \quad
\includegraphics[width = .47\textwidth]{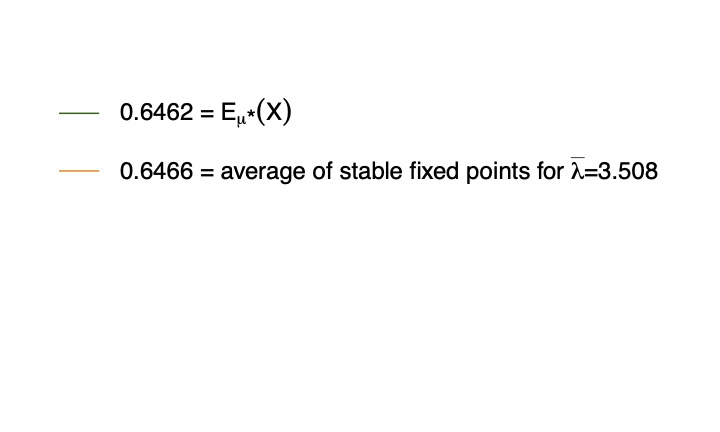}\\
\includegraphics[width = .47\textwidth]{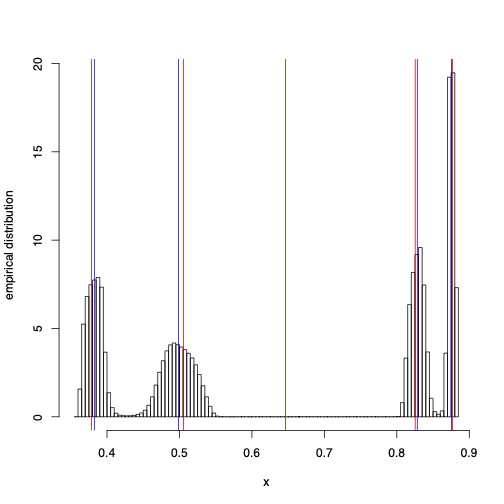} \quad
\includegraphics[width = .47\textwidth]{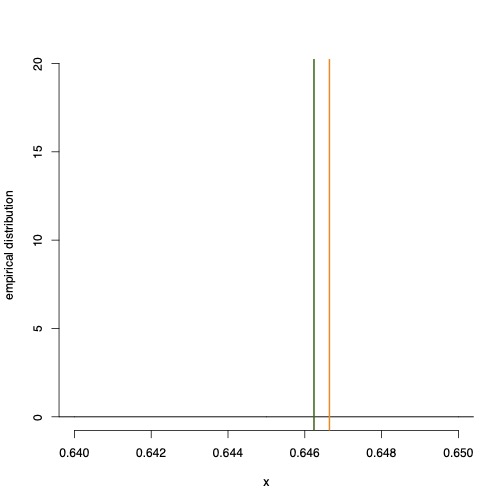}
\caption{Numerically estimated histogram of the $10000^{th}$ iterate of $P^*$, the Foias operator associated with $T(\vec{\lambda},x)$.  In this simulation,  the parameter value at each iteration is given by $ \lambda \sim {\mathcal U}[3.508 - \Delta \lambda, 3.508 + \Delta \lambda]$ with $\Delta \lambda = 0.024$; 20,000 initial $x$-values were drawn from a uniform distribution on $[0,1]$.    Since the system has a unique stable invariant measure, this histogram is an estimate of the distribution of the limiting invariant measure, $\mu^*$.  The right panel shows a magnification of the central portion of the histogram, highlighting that the mean of the empirical histogram (green line) is less than the average of the stable fixed points of the invariant distribution at the average parameter value, $\bar{\lambda} = 3.508$ (orange line). }
\label{HistFour}
\end{center}
\end{figure}

\section{Discussion}

\label{sec:discussion}

Previous studies have explored the  effect of randomness on the average behavior of  families of one-dimensional maps with stable fixed points \cite{Haskell2005, Hinow2016}.
In this paper we compared the expected value of the deterministic and stochastic logistic map in different periodic regimes. We showed that a unique stable invariant measure existed under the Foias operator associated with the stochastic logistic map in the parameter regions of interest.  This allowed us to interpret numerical experiments, and allowed us to compare the expected value of the stochastic logistic map with that of the deterministic map.

In the period one region, the average of the deterministic map, which is the fixed point of the map at the average parameter values, is greater than the expected value of $X$ under the unique stable invariant distribution of the stochastic map.  Since this distribution is stable, the long-term average of almost all orbits under $T(\vec{\lambda},x)$ is strictly less than the long-term average under the related deterministic map.  In terms of a growth model, if an environmental parameter that affects the birth rate of the population fluctuates randomly around a mean value $\bar{\lambda}$, then the average population will be smaller than in the deterministic model with parameter fixed at $\bar{\lambda}$.

These results for the period one region are similar to those in \cite{Haskell2005}, which apply to the stochastic Beverton-Holt model.   This is not surprising, since the proof of Theorem \ref{PeriodOneTheorem} uses only the fact that the logistic map is unimodal and concave down, features shared by the Beverton-Holt model.  In \cite{Hinow2016} a similar result is proved in greater generality, allowing for arbitrary distributions of the parameter.

In contrast to the period one region, in the period two region the average of the deterministic map is {\bf less} than the expected value of $X$ under the unique stable invariant distribution of the stochastic map.   The period two result is given in Theorem \ref{PeriodTwoTheorem} and formally proved in the Appendix.  The switch in behavior led us to explore what happens after further period doubling bifurcations.

Our empirical simulations (shown in Figure \ref{HistFour}) demonstrate that the behavior in the period four region reverts to that of the period one region.  In particular, it seems as though the average of the deterministic map is again greater than the expected value of $X$ under the unique stable invariant distribution of the stochastic map.  The oscillating behavior across the different regions leads us to the following ``flip-flop" conjecture.

\begin{conj} Flip-flop conjecture.

\begin{enumerate}
\item
For $\rho$ {\bf odd}, consider regions with period $2^\rho$.  Let $p_i(\bar{\lambda})$ be the stable period-$2^\rho$ points of the deterministic map with the average parameter value, $\bar{\lambda}$.
 
The expected value of $X$ under the stochastic logistic map is {\bf greater} than the average of the stable period-$2^\rho$ points of the deterministic map with the average parameter value,  $\bar{\lambda}$. 
 $$E_{\mu^*}[X] > \frac{1}{2^\rho} \sum_{i=1}^{2^\rho}p_i(\bar{\lambda}) $$

\item
For $\rho$ {\bf even}, consider regions with period $2^\rho$.  Let $p_i(\bar{\lambda})$ be the stable period-$2^\rho$ points of the deterministic map with the average parameter value, $\bar{\lambda}$.
 
The expected value of X under the stochastic logistic map is {\bf less} than the average of the stable period-$2^\rho$ points of the deterministic map with the average parameter value,  $\bar{\lambda}$. 
 $$E_{\mu^*}[X] < \frac{1}{2^\rho} \sum_{i=1}^{2^\rho}p_i(\bar{\lambda}) $$
\end{enumerate}

\end{conj}

We note that a renormalization framework can be used to explore the conjecture, since the convexity of the iterates: $S_\lambda(x)$, $S^2_\lambda(x)$, $S^4_\lambda(x), \dots$ at the fixed point switches as the number of iterates is doubled. Thus, it makes sense to conjecture that as we increase the period, the direction of the inequality switches back and forth, at least through the period doubling range of $\lambda $.

\begin{appendix}

\section{Details of the proof of Theorem \ref{PeriodTwoTheorem}}
\label{sec:appendix}

This Appendix contains the formal proof of the main result of the paper, Theorem \ref{PeriodTwoTheorem}.  The proof consists of a series of lemmas.  To guide the reader and to highlight the steps in the argument, we first provide an outline of the lemmas.\\

Recall Theorem \ref{PeriodTwoTheorem}:\\
 In the period two region, $\lambda \sim {\cal U}(a,b)$, with $[a,b] \subset (3, 1 + \sqrt{6}) $, the expected value of $X$ under $\mu^*$ (the stochastic logistic map) is {\bf greater} than the average of the stable period-2 points of the deterministic map with the average parameter value,  $\bar{\lambda}$. 
 $$E_{\mu^*}[X] > \frac{p(\overline{\lambda})+ q(\overline{\lambda})}{2} = \frac{\overline{\lambda} + 1}{2 \overline{\lambda}}$$

\begin{enumerate}
\item
First, we establish that, for small amplitudes of noise, the support of the invariant measure is contained in two disjoint intervals ({\bf Lemma \ref{LambdaClaim}}).

\item The expected value of $X$ with respect to the right hand measure (as defined in Equation (\ref{PeriodOneTheorem})) is the same as the expected value of the first iterate of the deterministic map at $\overline{\lambda}$ with respect to the left hand measure ({\bf Lemma \ref{Lemma1}}).\\

$$ E_{\mu_q^*}[X] = E_{\mu_p^*}[S_{\bar{\lambda}}(X)] $$

\item The left peak of the invariant measure, i.e., the expected value with respect to $\mu_p^*$, is larger than  the smaller point on the period-2 cycle of $S_{\bar{\lambda}}$ ({\bf Lemma \ref{Lemma3}}).
$$E_{\mu_p^*}[X] > p(\bar{\lambda}) $$

To prove Lemma \ref{Lemma3}: create functions $H$ and $F$, apply Lemma \ref{Lemma1} twice, then Lemma \ref{LambdaClaim} and use Jensen's inequality.

\item We define a function of $\Delta \lambda$ giving the variance of $X$ with respect to the measure $\mu_q^*$.  We show that the right hand derivative of the function $V(\Delta \lambda)$ is zero when $\Delta \lambda = 0$  ({\bf Lemma \ref{Lemma2}}).\\
$$V_{+}'(0)  =  \lim_{h \rightarrow 0^+} \frac{V(h + 0 )- V(0) }{h}  = 0 $$
We use this result to prove the next lemma.
\item  The change in the expected value of the left peak of the invariant distribution is greater in magnitude than the change in the expected value of the right peak at $\Delta \lambda = 0$ ({\bf Lemma \ref{Lemma4}}).
$$ \frac{dE_{\mu_p^*}[X]}{d \Delta \lambda} \ge - \frac{d E_{\mu_q^*} [X]}{d \Delta \lambda} $$

\item Which allows us to prove Theorem  \ref{PeriodTwoTheorem}.  By Lemma \ref{Lemma4}, for very small $\Delta \lambda \ge 0$, $E_{\mu_p^*}[X]$ increases faster than $E_{\mu_q^*}[X]$ decreases in $\Delta \lambda$, implying that their mean, $E[X]$, increases with $\Delta \lambda$.

\end{enumerate}
\paragraph*{Notation} We provide some definitions of notation that will be used throughout the proof.

\begin{eqnarray}
 \mbox{For }\lambda \sim {\mathcal U}[\bar{\lambda} - \Delta\lambda, \bar{\lambda} + \Delta\lambda], \ \ \ \ g(\lambda) = \begin{cases} \frac{1}{\bar{\lambda} + \Delta\lambda - (\bar{\lambda} - \Delta\lambda)} = \frac{1}{2\Delta\lambda }& \bar{\lambda} - \Delta\lambda \le \lambda \le \bar{\lambda} + \Delta\lambda  \\ 0 & \hbox{otherwise} \end{cases} . \label{eq:glambda}
\end{eqnarray}

For values of $\lambda$ greater than $3$, the logistic map has two unstable fixed points at 0 and $x_{\lambda}^* = 1 - \frac{1}{\lambda}$, and two period-2 points, $p_\lambda < q_\lambda$  at
\begin{equation} \label{eq:per2pts}  p(\lambda) =  \frac{\lambda + 1 - \sqrt{\lambda^2 - 2 \lambda - 3}}{2 \lambda}, 
\quad \quad q(\lambda) =  \frac{\lambda + 1 + \sqrt{\lambda^2 - 2 \lambda - 3}}{2 \lambda} .
\end{equation}
Note that the fixed points are solutions of the quadratic:
$$ \lambda x (1- x) - x = 0$$
while the period-2 points are solutions to the quartic:
$$ \lambda^2 x(1-x)(1 - \lambda x (1-x))  - x = 0 $$
that are {\it not} fixed points.

\begin{eqnarray}
p_{\pm} = p(\bar{\lambda} \pm \Delta \lambda) \quad q_{\pm} = q(\bar{\lambda} \pm \Delta \lambda) \label{eq:pqplusminus}
\end{eqnarray}

The period-two points for $\bar{\lambda} = 3.2$, $\Delta \lambda = 0.1$, are shown in Figure \ref{fig:twoiter}.  Note that $p_+ < p_-$ while $q_- < q _+$.  

\subsection{Proof of Theorem \ref{PeriodTwoTheorem}}
The proof of Theorem \ref{PeriodTwoTheorem} is divided up into several steps.  First, we establish that, for small amplitudes of noise, the support of the invariant measure is contained in two disjoint intervals.  This was discussed in  Section \ref{sec:period2}.  

Suppose the values of $\lambda$ are in some interval $[\bar{\lambda} - \Delta \lambda, \bar{\lambda} + \Delta \lambda]$.  Let $I_p$ be the image under $T(\vec{\lambda}, \cdot)$ of the interval of period-2 points, $[q(\bar{\lambda} - \Delta \lambda), q(\bar{\lambda} + \Delta \lambda)]$, and let $I_q$ be the image of the other interval of period-2 points: $[p(\bar{\lambda} + \Delta \lambda), p(\bar{\lambda} - \Delta \lambda)]$.  We will see that $I_p \cap I_q = \emptyset$.  
\begin{lemma}
\label{LambdaClaim}
For $\bar{\lambda} \in (3, 1 + \sqrt{6})$ with $\Delta \lambda$ such that $[\bar{\lambda} - \Delta \lambda, \bar{\lambda} + \Delta \lambda] \subset (3, 1 + \sqrt{6})$,  the support of the invariant distribution, $I$, is contained in the disjoint union $I_p \cup I_q$.
\end{lemma}
\noindent

\begin{proof}
If $\lambda \in  (\bar{\lambda} - \Delta \lambda, \bar{\lambda} + \Delta \lambda)$, with 
\begin{equation} \label{eq:lambdainterval} 
3 < \bar{\lambda} - \Delta \lambda < \bar{\lambda} + \Delta \lambda < 1 + \sqrt{6}
\end{equation} 
then the dynamical system given by 
$$ x_{t + 1} =  S_\lambda(x_t) =   \lambda x_t ( 1 - x_t) $$
has an asymptotically stable period-2 orbit \cite{Devaney} that attracts almost all orbits in $[0,1]$.   In other words, for a fixed value of $\lambda$,  for almost all $x_0 \in (0,1)$ with respect to Lebesgue measure, the limit set of $\{x_t \}_{t = 0}^\infty$ is $\{ p(\lambda), q(\lambda) \}$  as given in Equation \eqref{eq:per2pts}.
 This implies that almost all orbits of the stochastic map will approach a neighborhood of the set of period-2 points: $[p_+, p_-] \cup [q_-, q_+]$.   

Let $T(\vec{\lambda},x)$ be the stochastic logistic map with $\lambda \sim {\mathcal U}[\bar{\lambda} - \Delta \lambda, \bar{\lambda} + \Delta \lambda]$, and let $\mu^*$ be the unique stable invariant measure under $T$, as defined in Section \ref{sec:ergodicity}.  Denote the support of $\mu^*$ by $I$.  We will show that $I \subset I_p \cup  I_q$, where  are given by $I_p = [x_{p, \min}, x_{p, \max}]$ and $I_q = [x_{q, \min}, x_{q, \max}]$ as seen in Figure \ref{fig:disjoint}. We give an explicit description of $I_p$ and $I_q$ below.

For a fixed $\bar{\lambda}$ and $\Delta \lambda$ satisfying Equation \eqref{eq:lambdainterval}, let $p_\pm$ and $q_\pm$ be defined as in Equation \eqref{eq:pqplusminus}.  Remember that, if a process is ergodic, the support of the unique invariant measure contains all invariant sets of positive measure.  The stable period-2 points of $S_\lambda(x)$ for all $\lambda \in  (\bar{\lambda} - \Delta \lambda, \bar{\lambda} + \Delta \lambda)$ must be contained in the support of $\mu^*$ since these points are invariant under $S_\lambda(x)$ for each $\lambda \in (\bar{\lambda} - \Delta \lambda, \bar{\lambda} + \Delta \lambda)$.  Hence $I \supset [p_+, p_-] \cup   [q_-, q_+]$.  Figure \ref{fig:twoiter}  shows that $I$ can, in fact, contain more than the set of stable period-2 points. This follows from the observation that the function $S_{\lambda}(x)$ can be non-monotonic in the interval $[p_+, p_-]$,  and the functions are decreasing on the interval $[q_-, q_+]$ for $\lambda \in (\bar{\lambda} - \Delta \lambda, \bar{\lambda} + \Delta \lambda)$, so that $S_{\bar{\lambda} + \Delta \lambda}(q_-) > S_{\bar{\lambda} - \Delta \lambda}(q_-) = p_-$, and 
$S_{\bar{\lambda} - \Delta \lambda}(q_+) < S_{\bar{\lambda} + \Delta \lambda}(q_+) = p_+$

\begin{figure}[H]
\begin{center}
\includegraphics[scale=.75]{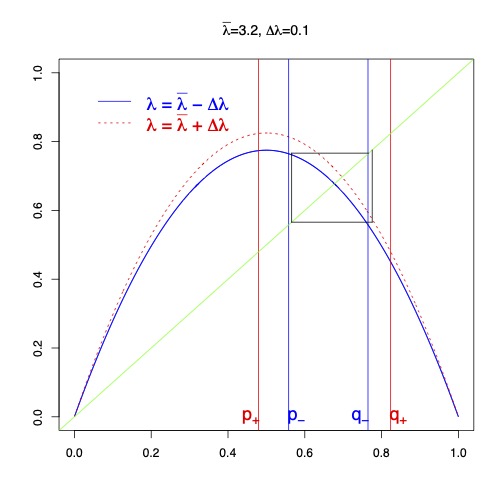}
\caption{Two iterates of the stochastic logistic map, $T(\vec{\lambda}, x)$  with initial value $x_0 = .7755$ and parameter distribution $\lambda \sim {\mathcal U}[3.1, 3.3]$, i.e., $\bar{\lambda} = 3.2, \Delta \lambda = .1$.  
}
\label{fig:twoiter}
\end{center}
\end{figure}

\begin{figure}[H]
\begin{center}
\includegraphics[width =  .9\textwidth, height =  5.5in]{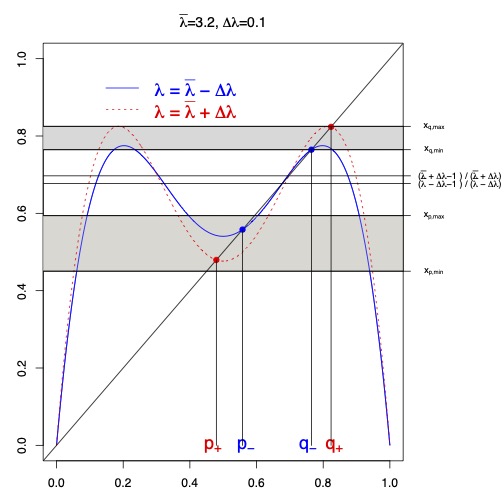}
\caption{The disjoint intervals $[x_{p,\min}, x_{p,\max}]$ and $[x_{q, \min}, x_{q, \max}]$ are the images of $[q_-, q_+]$ and $[p_+, p_-]$, respectively, under $S_{\lambda_2}(S_{\lambda_1})(x)$ for all $\lambda_{1,2} \in [\bar{\lambda} - \Delta \lambda, \bar{\lambda} + \Delta \lambda]$.
}
\label{fig:disjoint}
\end{center}
\end{figure}

For illustration, consider the case when $ 1 + \sqrt{5} \in (\bar{\lambda} - \Delta \lambda, \bar{\lambda} + \Delta \lambda)$, so that $p_+ \le \frac{1}{2} \le p_-$ and the logistic map is non-monotonic on the interval   $[p_+ , p_-]$.
Let $x_t \in  [p_+, p_-]$ and consider $x_{t+1} = T(\vec{\lambda},x_t)$.  We define $x_{q, \min}$ to be the minimum value of $x_{t+1}$, so that $x_{q,\min}$ is the smaller of $q_-$ and $S_{\bar{\lambda}-\Delta \lambda}(p_+)$, noting that it is possible that $x_{q, \min}$ can be smaller than $q_-$.   Similarly, in this case the maximum value of $x_{t+1}$ is 
$x_{q, \max} = S_{\bar{\lambda} + \Delta \lambda} \left({\frac{1}{2}}\right)= \frac{(\bar{\lambda} + \Delta \lambda)}{4} .$  

In general, define: 
\begin{eqnarray} \label{eq:xminxmax}
x_{p,\min} &=&  S_{\bar{\lambda} - \Delta \lambda}(q_+), \quad
\quad \quad \quad \quad x_{p,\max} = S_{\bar{\lambda} + \Delta \lambda}(q_-)
\\
\quad x_{q,\min} &=& \min\{q_-, S_{\bar{\lambda}-\Delta \lambda}(p_+)\},
\quad x_{q,\max} = \max \bigg\{ S_{\bar{\lambda} + \Delta \lambda} \bigg(\frac{1}{2}\bigg), S_{\bar{\lambda} + \Delta \lambda} (p_+), 
S_{\bar{\lambda} + \Delta \lambda} (p_-) \bigg\}   \nonumber 
\end{eqnarray}

Note that if $\Delta \lambda = 0$ and $\bar{\lambda} > 3$ then
\begin{eqnarray*}
x_{p, \max} &=& p(\bar{\lambda}) < \frac{\bar{\lambda} - 1}{\bar{\lambda}}\\
x_{q, \min} &=& q(\bar{\lambda}) > \frac{\bar{\lambda} - 1}{\bar{\lambda}}\\
\end{eqnarray*}
By continuity in $\Delta \lambda$, for small enough values of $\Delta \lambda$\footnote{An explicit upper bound for $\Delta \lambda$ can be calculated, as the relationships described can all be written out algebraically.  However, for our purposes, we only need the results to hold for some small enough $\Delta \lambda > 0$.  We omit the algebraic details for clarity of exposition.}, we see that the intervals $I_p$ and $I_q$ lie on either side of the non-zero fixed points corresponding to $\bar{\lambda} + \Delta \lambda$ and $\bar{\lambda} - \Delta \lambda$, as follows: 
\begin{eqnarray}
p_+ < p_- =  p(\bar{\lambda} - \Delta \lambda) \le x_{p,\max} &<& 
\frac{\bar{\lambda} - \Delta \lambda -  1}{\bar{\lambda} - \Delta \lambda} \nonumber \\
&<&  \frac{\bar{\lambda} + \Delta \lambda -  1}{\bar{\lambda} + \Delta \lambda} < x_{q, \min} \le
  q_- = q(\bar{\lambda} - \Delta \lambda) < q_+ 
 \label{eq:pqinequal}
 \end{eqnarray}
so the intervals $I_p$ and $I_q$ are indeed disjoint.  
See Figure \ref{fig:disjoint} which displays the ordered values given in Equation \eqref{eq:pqinequal}.
\end{proof}

Remark: the supports of $\mu_p^*$ and $\mu_q^*$ are contained in $I_p$ and $I_q$, respectively.  Lemma \ref{LambdaClaim} therefore implies that all points in the support of $\mu_p^*$ are mapped to the support of $\mu_q^*$ under one iterate of $T(\vec{\lambda}, x)$, so that, by the definition of the Foias operator, $P^*: \mu_p^* \mapsto \mu_q^*$ and vice-versa.  Therefore, by construction, $\mu^*(support(\mu_q^*)) = \mu^*(support(\mu_p^*)) = \frac{1}{2}$.

\begin{lemma}  \label{Lemma1}
$$ E_{\mu_q^*}[X] = E_{\mu_p^*}[S_{\bar{\lambda}}(X)] .$$
when $\lambda \sim {\mathcal U}[\bar{\lambda} - \Delta\lambda, \bar{\lambda} + \Delta\lambda]$ with density $g(\lambda)$ as given in Equation \eqref{eq:glambda}.
\end{lemma}

\begin{proof}
We note that $P^* \mu_p^*(x) = \mu_q^*(x)$, where $P^*$ is the Foias operator given in Equation \eqref{eq:PF}.  Therefore, we have:
$$ E_{\mu_q^*}[X] = \int_0^1 x \mu_q^*(x) \, dx = \int_0^1 x P^* \mu_p^*(x) \, dx = \int_0^1 x \int_{V_x} \mu_p^*(y) g\left({\frac{x}{y(1-y)}}\right) \frac{1}{y(1-y)} \, dy \, dx , $$
where the set $V_x$ is $\{ (\lambda, y) : \lambda y (1 - y) = x \}$.  Changing the order of integration, we have that:
$$ E_{\mu_q^*}[X] = \int_0^1 \left[{ \int_0^1 x g\left({\frac{x}{y(1-y)}}\right)  \, dx }\right] \frac{1}{y(1-y)} \mu_p^*(y) \, dy .$$
The $\lambda$'s are uniformly distributed on $[\bar{\lambda} - \Delta\lambda, \bar{\lambda} + \Delta\lambda]$ so that $g(\lambda)$ is as given in Equation \eqref{eq:glambda}.  Using $g(\lambda)$ in the expression for $E_{\mu_q^*}[X]$  and evaluating the integral we get:
\begin{eqnarray*}
E_{\mu_q^*}[X] &=& \int_0^1 \left[{ \frac{x^2}{2} }\right\vert_{(\bar{\lambda} - \Delta\lambda) y(1-y)}^{(\bar{\lambda} + \Delta\lambda) y(1-y)} \, \frac{\mu_{p^*}(y)}{y(1-y) (2 \Delta\lambda)}\, dy \\
&=& \int_0^1 \frac{\left({(\bar{\lambda} + \Delta\lambda)^2 - (\bar{\lambda} - \Delta\lambda)^2}\right) \left({ (y(1-y))^2}\right)}{2} \,   \frac{\mu_{p^*}(y)}{y(1-y) (2\Delta\lambda )}\, dy \\
&=& \bar{\lambda} \int_0^1 y(1-y) \mu_{p^*} (y) \, dy
\end{eqnarray*}

We can rewrite $E_{\mu_q^*}[X]$ as:
$$ E_{\mu_q^*}[X] = \bar{\lambda} \left({E_{\mu_{p^*}}[X] - E_{\mu_p^*}[X^2] }\right) = \bar{\lambda} E_{\mu_{p^*}} [S_1(X)]
= E_{\mu_p^*}[S_{\bar{\lambda}}(X)]
.$$
\end{proof}


\begin{lemma}{The left peak of the invariant measure, i.e., the expected value with respect to $\mu_p^*$, is larger than  the  smaller point on the period-2 cycle of $S_{\bar{\lambda}}$.} \label{Lemma3}
$$E_{\mu_p^*}[X] > p(\bar{\lambda}). $$

Furthermore, the right peak, i.e., the expected value with respect to $\mu^*_{q}$, is less than the larger period-2 point of $S_{\bar{\lambda}}$:
$$ E_{\mu^*_q} < q(\bar{\lambda}).$$

\end{lemma} 

\begin{proof}
Note that we are assuming that $\Delta \lambda > 0$ in order to get a strict inequality.  To prove Lemma \ref{Lemma3}, we introduce a new function, $H_{\bar{\lambda}}(x)$, which helps us put bounds on the second iterate of $S_{\bar{\lambda}}(x)$.

Fix values of $\bar{\lambda}$ and $\Delta \lambda$  in the period-2 regime: $\bar{\lambda} \pm \Delta \lambda \in (3, 1 + \sqrt{6} )$. 

Consider the functions  $F_{\bar{\lambda}}(x) = (S_{\bar{\lambda}} \circ S_{\bar{\lambda}})(x) - x$, and
$H_{\bar{\lambda}}(x) =  \bar{\lambda} h_{\bar{\lambda}, \epsilon}(x) - x$ where 
\begin{equation} \label{eq-h(x)}
 h_{\bar{\lambda}, \epsilon}(x) = \bar{\lambda} x (1-x) -  \bar{\lambda}^2 \left({ x (1-x)}\right)^2   + \epsilon
 \end{equation}
 for some $\epsilon > 0$.  From here on we drop the subscripts on $h$ for clarity.
 
The function $F_{\bar{\lambda}}(x)$ has four zeros. Two of the zeros are the two fixed points of $S_{\bar{\lambda}}(x)$: $0$ and $\frac{\bar{\lambda}-1}{\bar{\lambda}}= x^*(\bar{\lambda})$; two zeros are given by the period-2 orbit:  $p(\bar{\lambda})$ and $q(\bar{\lambda})$, as defined in Equation (\ref{eq:per2pts}).
We see that the function $H$ is just a shift upward of the function $F$ by $\bar{\lambda} \epsilon$: $H(x) = F(x) + \bar{\lambda} \epsilon$.   If $\epsilon$ is small enough, $H$ will have the same number of zeros as $F$.  If we label the zeros of $H$ as: $z_H < p_H < x^*_H < q_H$, we have (see Figure \ref{fig:HandF}):
\begin{equation}
    \label{eq:pqinequal2}
z_H <  0  < p(\bar{\lambda} ) < p_H  < x^*_H  < x^*(\bar{\lambda}) = \frac{\bar{\lambda}-1}{\bar{\lambda}} < q(\bar{\lambda} )  < q_H < 1 .
\end{equation}

Now, fix $\bar{\lambda}$, and we will drop the subscript $\bar{\lambda}$ on all the functions for readability.
We need the following to complete the argument.
\begin{claim}
The second derivative of $h(x)$, $h''(x)$,  is strictly positive for $x \in I_p$, and $h''(x)$ is strictly negative for $x \in I_q$, when $\Delta \lambda$ is small enough.
\label{h2claim}
\end{claim}

By Claim \ref{h2claim}, the function $h$ is concave up on $I_p$, the support for $\mu^*_p(x)$, and $h$ is concave down on $I_q$, the support for $\mu^*_q(x)$.  
 This implies, by Jensen's Inequality, that 
\begin{equation} \label{eq:Jensens1}
E_{\mu_p^*}[h(X)] > h(E_{\mu_p^*}[X]) \Rightarrow 
E_{\mu_p^*}[h(X)] =  h(E_{\mu_p^*}[X]) + \delta \quad \hbox{for some } \delta > 0. 
\end{equation}
and
\begin{equation} \label{eq:Jensens2}
E_{\mu_q^*}[h(X)] < h(E_{\mu_q^*}[X]) \Rightarrow 
E_{\mu_q^*}[h(X)] =  h(E_{\mu_q^*}[X]) - \delta^* \quad \hbox{for some } \delta^* > 0. 
\end{equation}
 
We first prove Claim \ref{h2claim} and then continue with the proof of Lemma \ref{Lemma3}.

\begin{proof}
(of Claim \ref{h2claim}) By direct calculation we see that 
$$h''(x) = -2( \bar{\lambda} + \bar{\lambda}^2) + 12 \bar{\lambda}^2 (x - x^2).$$
Thus, $h''(x) > 0$ for $x \in (r_1, r_2)$ where
$$ r_{1,2} = \frac{1}{2} \pm \frac{1}{2} 
\sqrt{1 - \frac{2}{3} \left({\frac{1}{\bar{\lambda}} + 1 }\right) }$$
We will establish the following sequence of inequalities, which holds for sufficiently small values of $\Delta \lambda$:
\begin{equation}
    \label{eq:r1r2inequality}
r_1 < x_{p,\min} < x_{p,\max} < r_2 < x^*_{\bar{\lambda}} < x_{q,\min}.
\end{equation}

Together with the inequalities in Equation (\ref{eq:pqinequal}), the inequalities in Equation (\ref{eq:r1r2inequality}) prove Claim \ref{h2claim}.  We therefore split the justification of Claim \ref{h2claim} into the three relevant inequalities of Equation (\ref{eq:r1r2inequality}):

\begin{enumerate}
\item $r_1 < x_{p,\min} = S_{\bar{\lambda}-\Delta \lambda}(q_+).$ \\
{\it Proof:} Since $\bar{\lambda} > 3$, we have $\frac{1}{\bar{\lambda}} < \frac{1}{3}$, and so
$$r_{1} = \frac{1}{2} - \frac{1}{2} 
\sqrt{1 - \frac{2}{3} \left({\frac{1}{\bar{\lambda}} + 1 }\right) } < \frac{1}{3}$$
We know that $3 < \bar{\lambda} - \Delta \lambda$ and $\frac{1}{2} < q_+ < q(1 + \sqrt{6}),$ which implies that $S_3(q(1 + \sqrt{6})) < S_{\lambda_-}(q_+).$  (This is true because $S_\lambda(x)$ is increasing in $\lambda$ and decreasing in $x > \frac{1}{2}.$)  By direct calculation, $\frac{1}{3} < S_3(q(1 + \sqrt{6})).$  Therefore, $r_1 < \frac{1}{3} < S_3(q(1 + \sqrt{6})) < S_{\bar{\lambda}-\Delta \lambda}(q_+) = x_{p,\min} .$

\item $r_2 > x_{p,\max} = S_{\bar{\lambda}+\Delta \lambda}(q_-)$.\\
{\it Proof:} Since $\bar{\lambda} > 3$, we have $\frac{1}{\bar{\lambda}} < \frac{1}{3}$, and so
$$r_{2} = \frac{1}{2} + \frac{1}{2} 
\sqrt{1 - \frac{2}{3} \left({\frac{1}{\bar{\lambda}} + 1 }\right) } > \frac{2}{3} \Rightarrow r_2 = \frac{2}{3} + \eta, \quad \hbox{ for some } \eta > 0
$$

We note that $\frac{2}{3} = p(3) > p_-$, since the values of $p$ decrease as $\lambda$ increases.  Therefore, for $\Delta \lambda$ small enough, we can ensure that 
$p(3) + \eta > \frac{\bar{\lambda} + \Delta \lambda}{\bar{\lambda} - \Delta \lambda} \cdot p_-$ which gives:
$$r_2 > \frac{\bar{\lambda} + \Delta \lambda}{\bar{\lambda} - \Delta \lambda} \cdot p_- = (\bar{\lambda} + \Delta \lambda)
\frac{
(\bar{\lambda} - \Delta \lambda)q_-(1-q_-)}{\bar{\lambda} - \Delta \lambda} = S_{\bar{\lambda}+\Delta \lambda}(q_-).
$$
\item $r_2 < x^*_{\bar{\lambda}} < x_{q,\min}$ \\

{\it Proof:} We show that $r_2 < x^*_{\bar{\lambda}}$, and the remaining inequality follows from Equation \eqref{eq:pqinequal}.  
\begin{eqnarray*}
     x^*_{\bar{\lambda}} > r_2 &\Leftrightarrow& 
1 - \frac{1}{\bar{\lambda}} > \frac{1}{2} \left({1 + \sqrt{1 - \frac{2}{3} \left({ \frac{1}{\bar{\lambda}} +1}\right)} }\right)  \\
&\Leftrightarrow& 
1 - \frac{2}{\bar{\lambda}} > \sqrt{1 - \frac{2}{3} \left({ \frac{1}{\bar{\lambda}} +1}\right)} \\
&\Leftrightarrow& 
\frac{1}{\bar{\lambda}}+1 > \frac{6}{\bar{\lambda}} - \frac{6}{\bar{\lambda}^2} \\
&\Leftrightarrow& 
\bar{\lambda}^2 > 5 \bar{\lambda} - 6 \\
\end{eqnarray*}
which holds if $\bar{\lambda} > 3$.  Therefore, since the fixed point $x^*$ increases as $\lambda$ increases, we have:
$$r_2 < x^*_{\bar{\lambda}} < x^*_{\bar{\lambda} + \Delta \lambda} < x_{q,\min}$$
\end{enumerate}

 Therefore, $h''(x) > 0$ for all $x \in I_p$, and $h''(x) < 0$ for all $x \in I_q$.
 \end{proof}

Now we return to the proof of Lemma \ref{Lemma3}.
By applying Lemma \ref{Lemma1} twice and using Equation \eqref{eq:Jensens1}  we get:
$$ E_{\mu_p^*}[X] = E_{\mu_q^*}[S_{\bar{\lambda}}(X)] = E_{\mu_p^*}[S^2_{\bar{\lambda}}(X)]
= \bar{\lambda}E_{\mu_p^*}[h(X)] - \bar{\lambda} \epsilon
= \bar{\lambda} h(E_{\mu_p^*}[X]) + \bar{\lambda} (\delta - \epsilon).$$

Rearranging the equation, we have 
\begin{equation}
\bar{\lambda} h(E_{\mu_p^*}[X]) - E_{\mu_p^*}[X] = \bar{\lambda}(\epsilon - \delta)  = H(E_{\mu_p^*}[X]) .
\end{equation}
\noindent
Since $\epsilon$ can be arbitrarily small, 
and since $\delta$ depends only on $h''$ which is independent of $\epsilon$, we can make sure that  $0 < \epsilon < \delta$, so that $H(E_{\mu_p^*}[X]) < 0 $.  This means that $E_{\mu_p^*}[X]$ is either less than $z_H < 0$, greater than $q_H$, or in the interval $(p_H, x^*_H)$ (see Figure \ref{fig:HandF}).   By Lemma \ref{LambdaClaim}, it is not possible for $E_{\mu_p^*}[X]$ to be in the support of $\mu_q^*$, which implies that $ p_H < E_{\mu_p^*}[X] < x^*_H$. Because $H(x) > F(x)$,  $p(\bar{\lambda}) < p_H$, from which we conclude that
\begin{equation}
\label{eq:qbias}
E_{\mu_p^*}[X] >  p(\bar{\lambda})
\end{equation}
and, as promised, the ``left peak" shifts up with noise added, i.e., when $\Delta \lambda > 0$.

To prove the statement about the right peak, we repeat the argument with the function $H$  replaced by the function $\hat{H} = F - \bar{\lambda} \hat{\epsilon}$, where $\hat{\epsilon}> 0$.  In other words, $\hat{H}$ is the function $F$ shifted {\it down} by a small amount.  In the argument, we replace the function $h$ by $\hat{h}$, where
$$\hat{h}_{\bar{\lambda},\epsilon} = \bar{\lambda}x(1-x)-\bar{\lambda}^2(x(1-x))^2 - \hat{\epsilon}.$$
The remainder of the argument that $E_{\mu_q^*}[X] <  q(\bar{\lambda})$ follows directly from the previous full proof of $E_{\mu_p^*}[X] > p(\bar{\lambda})$ with the inequalities reversed.
\end{proof}

\begin{figure}[H]
\begin{center}
\includegraphics[width = .8\textwidth]{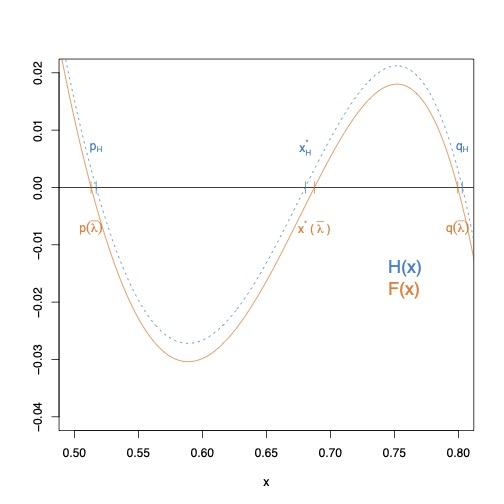}
\end{center}
\caption{$H(x)$ (dashed blue line) is the function which shifts $F(x)$ (solid orange line) up by a small amount, $\epsilon > 0$.  The curvature of the two functions is identical, but the zeros of $H(x)$ are all shifted. In particular, the zero at $p_H$ is to the right of $p(\bar{\lambda})$, which is the zero of $F$ corresponding to the smaller period-2 point of $S_{\bar{\lambda}}$.
}
\label{fig:HandF}
\end{figure}

In this part of the discussion,  we fix $\bar{\lambda}$ and consider values of the parameter, $\lambda$, uniformly distributed between $\bar{\lambda} - \Delta \lambda$ and $\bar{\lambda} + \Delta \lambda$.  We want to emphasize here and in what follows that the measures $\mu_p^*$ and $\mu_q^*$ are functions of $\Delta \lambda$,  defined in Equation \eqref{inv_measures}.   In particular, we define a function:
\begin{eqnarray}
V(\Delta \lambda) = E_{\mu_q^*(\Delta \lambda)} [X^2] - E_{\mu_q^*(\Delta \lambda)}[X]^2. \label{eq:vfunc}
\end{eqnarray}
That is, $V(\Delta \lambda)$ gives the variance of $X$ with respect to the measure $\mu_q^*$, which is a function of $\Delta \lambda$.  The variance function, $V(\Delta \lambda)$ measures the difference between $E_{\mu_q^*(\Delta \lambda)} [X^2]$ and $E_{\mu_q^*(\Delta \lambda)}[X]^2$ as a function of $\Delta \lambda$.

The technical lemma, Lemma \ref{Lemma2},  is used to prove Lemma \ref{Lemma4}.  We show directly that the right-hand derivative of the function $V(\Delta \lambda)$ is zero when $\Delta \lambda = 0 $.  To show this, we us a very rough bound of the variance as the squared difference between the maximum and minimum values of the support of $\mu^*_q$.

\begin{lemma} \label{Lemma2}
$$V_{+}'(0)  =  \lim_{h \rightarrow 0^+} \frac{V(h + 0 )- V(0) }{h}  = 0 $$
\end{lemma} 
\begin{proof}
We first note that $V(0) = 0$.  

We consider the support of $\mu_q^*(\cdot, \Delta \lambda, \bar{\lambda}) = (x_{q,\min}, x_{q,\max} )$, the right peak of the invariant distribution, $\mu^*$,  in three distinct cases, and prove the lemma for each case separately.

\begin{itemize}
\item[Case 1]
If $\bar{\lambda} > 1 + \sqrt{5}$ and $\Delta \lambda$ is sufficiently small, in which case $x_{q,\min} = S_{\bar{\lambda} - \Delta \lambda} (p_+)$
and $x_{q. \max} = S_{\bar{\lambda} + \Delta \lambda}(p_-) $.
\item[Case 2]
If $\bar{\lambda} < 1 + \sqrt{5}$ and $\Delta \lambda$ is sufficiently small, in which case $x_{q, \min} = q_-$ and  $x_{q,\max} = q_+$.
\item[Case 3]
If $\bar{\lambda} = 1 + \sqrt{5}$ in which case 
 $x_{q, \min} = q_-$ and
 $x_{q,\max} = S_{\bar{\lambda} + \Delta \lambda}(\frac{1}{2})$.
\end{itemize}

\begin{itemize}
\item[Case 1]
We start with $\bar{\lambda} > 1 + \sqrt{5}$, which means that $p(\bar{\lambda}) < \frac{1}{2}$. For small enough $\Delta \lambda$, it follows that $p_+ < p_- < \frac{1}{2}$, i.e., $S_{\lambda}$ is monotonically increasing on the interval $[p_+, p_-]$ for all $\lambda \in [\bar{\lambda} - \Delta \lambda, \bar{\lambda} + \Delta \lambda]$.  This means that, for $x_t \in [p_+, p_1]$, the smallest possible value of 
$x_{t+1}$ is $S_{\bar{\lambda} - \Delta \lambda}(p_+) \triangleq x_{q,\min}$ and the largest possible value of $x_{t+1}$ is $S_{\bar{\lambda} + \Delta \lambda}(p_-) \triangleq x_{q,\max}$. 

 It follows that the squared difference of  of the length of the interval from $x_{q, min}$ to $x_{q, max}$ gives a (rather crude) upper bound for the variance, i.e., $V(h) \le (x_{q,\max} - x_{q,\min})^2$.   Thus:
$$ \lim_{h \rightarrow 0^+} \frac{V(h)}{h} \le \lim_{h \rightarrow 0^+} \frac{\left({
S_{\bar{\lambda} - \Delta \lambda} (p_+) - S_{\bar{\lambda} + \Delta \lambda}(p_-)
}\right)^2}{h}  $$

Writing $\Delta \lambda  = h$,  and expressing everything in terms of $q_+$ and $q_-$:

\begin{eqnarray}
(
S_{\bar{\lambda} - \Delta \lambda} (p_+) &-& S_{\bar{\lambda} + \Delta \lambda}(p_-)
)^2 \nonumber \\
&=&
\left({ (\bar{\lambda} + h )p_- (1 - p_-) - ( \bar{\lambda} - h) p_+ (1 - p_+)  }\right)^2 \nonumber \\
&=&
\left({ (\bar{\lambda} - h)  p_- (1 - p_-) + 2 h p_- (1 - p_-) 
- (\bar{\lambda} + h )p_+ (1 - p_+) 
+ 2 h p_+ (1 - p_+) }\right)^2 \nonumber \\
&=&
 \left({  (q_+ - q_-) -  2 h
\left[{ \frac{q_-}{\bar{\lambda} - h} + \frac{q_+}{\bar{\lambda} + h} }\right] }\right)^2  														\nonumber  \\
&=&
(q_+- q_-)^2 
- 4 h (q_+ - q_-)\left[{ \frac{q_-}{\bar{\lambda} - h} + \frac{q_+}{\bar{\lambda} + h} }\right] + \nonumber \\
&& \hspace*{2.1cm} 4 (h)^2 \left[{ \frac{q_-}{\bar{\lambda} - h} + \frac{q_+}{\bar{\lambda} + h} }\right] ^2  \label{eq:Vnum}
\end{eqnarray}

We now divide both sides of Equation \eqref{eq:Vnum}  by $h$ and take the limit as $h$ goes to zero:
\begin{equation} \label{eq:dVdLambda}
\lim_{h \rightarrow 0^+} \frac{V(h)}{h} \le
\lim_{h \rightarrow 0^+}\left({
\frac{(q_+ - q_-)^2 }{h} - 4  (q_+ - q_-)\left[{ \frac{q_-}{\bar{\lambda} - h} + \frac{q_+}{\bar{\lambda} + h} }\right] 
 + 4 (h) \left[{ \frac{q_-}{\bar{\lambda} - h} + \frac{q_+}{\bar{\lambda} + h} }\right] ^2 }\right)
\end{equation}

Since $\lim_{h \rightarrow 0^+} q_- = \lim_{h \rightarrow 0^+} q_+$, we see that the middle term of Equation \eqref{eq:dVdLambda} goes to zero, and the last term goes to zero because $h \rightarrow 0$.  Expanding the first term using Equation \eqref{eq:per2pts} gives:
 $$ q_+ - q_- = \frac{h}{(\bar{\lambda} + h)(\bar{\lambda} - h)} - \frac{(\bar{\lambda} + h)\sqrt{(\bar{\lambda} - h)^2 - 2 (\bar{\lambda} - h) - 3}}{2 (\bar{\lambda} + h)(\bar{\lambda} - h)}
 + \frac{(\bar{\lambda} - h)\sqrt{(\bar{\lambda} + h)^2 - 2 (\bar{\lambda} + h) - 3}}{2 (\bar{\lambda} + h)(\bar{\lambda} - h)}
 $$

 Squaring and dividing by $h$:
$$\frac{(q_+ - q_-)^2}{h} = \mbox{(A)} + \mbox{(B)} + \frac{\mbox{(C)}}{\mbox{(D)}}$$

where
\begin{eqnarray*}
\mbox{(A)} &=& \frac{h}{(\bar{\lambda} + h)^2 (\bar{\lambda} - h)^2}\\
\mbox{(B)} &=& \frac{1}{(\bar{\lambda} + h)(\bar{\lambda} - h)} \left[{\frac{(\bar{\lambda} - h)\sqrt{(\bar{\lambda} + h)^2 - 2 (\bar{\lambda} + h) - 3}}{(\bar{\lambda} + h)(\bar{\lambda} - h)} - \frac{(\bar{\lambda} + h)\sqrt{(\bar{\lambda} - h)^2 - 2( \bar{\lambda} - h) - 3}}{(\bar{\lambda} + h)(\bar{\lambda} - h)} }\right] \\
\mbox{(C)} &=& (\bar{\lambda} - h)^2 ((\bar{\lambda} + h)^2 - 2 (\bar{\lambda} + h) - 3) +  (\bar{\lambda} + h)^2 ((\bar{\lambda} - h)^2 - 2 (\bar{\lambda} - h) - 3) - \\
&& 2( \bar{\lambda} + h)(\bar{\lambda} - h) \sqrt{ ((\bar{\lambda} + h)^2 - 2 (\bar{\lambda} + h) - 3)((\bar{\lambda} - h)^2 - 2( \bar{\lambda} - h) - 3) }\\
 \mbox{(D)} &=& 4 h(\bar{\lambda} + h)^2 (\bar{\lambda} - h)^2
\end{eqnarray*}

Taking the limit as $h$ goes to zero on both sides, we see that (A) is zero because of the $h$ in the numerator; (B) is zero, since $\bar{\lambda} - h \rightarrow \bar{\lambda} + h \rightarrow \bar{\lambda}$ as $h \rightarrow 0^+$.  We can see that $\frac{\mbox{(C)}}{\mbox{(D)}}$ also goes to zero as $h \rightarrow 0^+$ by applying l'Hopital's rule.  
Therefore, $\displaystyle{\lim_{h \rightarrow 0^+} \frac{(q_+ - q_-)^2 }{h}  = 
\lim_{h \rightarrow 0^+} \frac{V(h)}{h} }= 0$.

Since the variance is always positive, we have:
$$ 0 \le \lim_{h \rightarrow 0^+} \frac{V(h)}{h} \le \lim_{h \rightarrow 0^+} \frac{(x_{q, max} - x_{q, min})^2 }{h} = 0 \Rightarrow V'_+(0) = \lim_{h \rightarrow 0^+} \frac{V(h)}{h}  = 0 $$
when $\bar{\lambda} > 1 + \sqrt{5}$.

\item[Case 2]
Next, when $\bar{\lambda} < 1 + \sqrt{5}$,  for small $\Delta \lambda$, the function 
$S_{\lambda}(x)$ is increasing on the interval $[p_+, p_-]$ for all $\lambda \in [\bar{\lambda} - \Delta \lambda,\bar{\lambda} + \Delta \lambda]$.  It follows that  $x_{q,\min} =  q_-$ and $x_{q,\max} = q_+$ so the proof simplifies:
$$ \lim_{h \rightarrow 0^+} \frac{V(h)}{h} \le \lim_{h \rightarrow 0^+} \frac{(q_+ - q_-)^2 }{h} .$$
It has already been shown that $$\lim_{h \rightarrow 0^+} \frac{(q_+ - q_-)^2 }{h} = 0,$$
therefore
$$ 0 \le \lim_{h \rightarrow 0^+} \frac{V(h)}{h} \le \lim_{h \rightarrow 0^+} \frac{(q_+ - q_-)^2 }{h} = 0 \Rightarrow V'_+(0) = \lim_{h \rightarrow 0^+} \frac{V(h)}{h}  = 0 $$
when $\bar{\lambda} < 1 + \sqrt{5}$.

\item[Case 3] Suppose $\bar{\lambda} = 1 + \sqrt{5}$, so that the maximum of the function coincides with the lower period-2 point when the parameter is $\bar{\lambda}$, i.e., $p(\bar{\lambda}) = \frac{1}{2}$.   We claim that in this case
\begin{equation} \label{eq:case3claim}
\lim_{h \rightarrow 0^+} \frac{V(h)}{h} \le \lim_{h \rightarrow 0^+} \frac{(S_{\bar{\lambda}+h}(1/2) - q_-)^2 }{h} .
\end{equation}
To see why this is true, we
let $f(x) = p(x) = \frac{x+1 - \sqrt{x^2 - 2x -3}}{2x}$ be the the smaller period-2 point when $\lambda = x$.   We are interested in values of $x$ near $1 + \sqrt{5}$ that are in the stable period-2 regime, i.e., $3 < x < 1 + \sqrt{6}$.  We see that
$$f'(x) = \frac{-x -3 - \sqrt{x^2 - 2x - 3}}{2x^2 \sqrt{x^2 - 2x -3}} < 0.$$
Furthermore,
$$f''(x) = 
  \frac{x^3 + (\sqrt{x^2 - 2 x - 3} + 3) x^2 - (2 \sqrt{x^2 - 2 x - 3} + 9) x - 3 (\sqrt{x^2 - 2 x - 3} + 3)}
 {x^3 (x^2 - 2 x - 3)^{3/2}} >  0 \quad \hbox{ when } x > 3 .
$$
We conclude that $f(x)$ has negative slope and is concave up for $x$ close to $1 + \sqrt{5}$, which means that $f(1 + \sqrt{5} - \Delta \lambda) $ is farther from $f(1 + \sqrt{5})$ than the distance from $f(1 + \sqrt{5} + \Delta \lambda )$ to $f(1 + \sqrt{5})$.
In other words,  $p_- - \frac{1}{2} > \frac{1}{2} - p_+$.  This implies that $$x_{q,\min} = S_{\bar{\lambda} - \Delta \lambda}\left({p(\bar{\lambda} - \Delta \lambda)}\right) = q_-  \quad. $$
We know that $x_{q,\max}$ is the image of $\frac{1}{2}$ under $S_{\bar{\lambda} + \Delta \lambda}$.  Letting $\Delta \lambda = h$ we have:
$$ V(h) \le (S_{\bar{\lambda}+h}(1/2) - q_-)^2 $$
and Equation \eqref{eq:case3claim} is verified.  

We now calculate the limit:
$$ \lim_{h \rightarrow 0^+} \frac{V(h)}{h} \le \lim_{h \rightarrow 0^+} \frac{(S_{\bar{\lambda}+h}(1/2) - q_-)^2 }{h} .$$
Notice that the numerator in the second limit is zero, since, as $h \rightarrow 0, p_- \rightarrow \frac{1}{2}$ 
$$ \Rightarrow
q_- =  S_{\bar{\lambda}-h}(p_-) \rightarrow S_{\bar{\lambda}}(1/2) = \lim_{h \rightarrow 0} S_{\bar{\lambda} + h}(1/2) .$$
Therefore, we apply l'Hopital's rule to find:
$$
 \lim_{h \rightarrow 0^+} \frac{(S_{\bar{\lambda}+h}(1/2) - q_-)^2 }{h}
 =
  \lim_{h \rightarrow 0^+} 
  2\left({S_{\bar{\lambda}+h}(1/2) - q_-}\right)
   \frac{d}{dh} \left({S_{\bar{\lambda}+h}(1/2) - q_-}\right)
  $$
This limit is zero as long as the last derivative term is finite.  We verify this:
$$
\frac{d}{dh}\left({S_{\bar{\lambda}+h}(1/2) - q_-}\right) =
\frac{d}{dh} \left[{
 \frac{\bar{\lambda}+h}{4}  - 
\frac{\bar{\lambda} - h + 1 -
 \sqrt{(\bar{\lambda}-h)^2 - 2(\bar{\lambda}- h) - 3 }}
 {2(\bar{\lambda}-h) }
 }\right]
 $$
$$
=  \frac{1}{4} + \frac{1}{4(\bar{\lambda} -h)^2}
 \left[{
 1 + \frac{h - \bar{\lambda} + 1}{2  \sqrt{(\bar{\lambda}-h)^2 - 2(\bar{\lambda}- h) - 3 }} 2 (\bar{\lambda} - h) 
 + 2 \left({ \bar{\lambda} - h + 1 -  \sqrt{(\bar{\lambda}-h)^2 - 2(\bar{\lambda}- h) - 3 } }\right)
 }\right]
$$
We now take the limit as $h \rightarrow 0$, remembering that 
$\bar{\lambda} = 1 + \sqrt{5}$, so that 
$\sqrt{\bar{\lambda}^2 - 3 \bar{\lambda} = 3} = 1$, which gives:
\begin{eqnarray*}
\lim_{h \rightarrow 0} 
\frac{d}{dh}\left({S_{\bar{\lambda}+h}(1/2) - q_-}\right) &=&
\frac{1}{4} + \frac{1}{4 \bar{\lambda}^2} 
\left[{
1 + \bar{\lambda}(1 - \bar{\lambda}) + 2 \bar{\lambda}
}\right] 
\\ &=&
\frac{1}{4} + \frac{1}{4 \bar{\lambda}^2} 
\left[{
1 - \bar{\lambda}^2 + 3 \bar{\lambda} 
}\right]
\\ &=&
\frac{1}{4} \left({ \frac{1 + 3 \bar{\lambda}}{\bar{\lambda}^2} 
}\right) < \infty
\end{eqnarray*}
Therefore, by l'Hopital's rule, the limit is zero, and
$$ \lim_{h \rightarrow 0^+} \frac{V(h)}{h} \le \lim_{h \rightarrow 0^+} \frac{(S_{\bar{\lambda}+h}(1/2) - q_-)^2 }{h} = 0  $$
\end{itemize}
\end{proof}

\begin{lemma}{The change in the expected value of the left peak of the invariant distribution is greater in magnitude than the change in the expected value of the right peak.}  \label{Lemma4}
$$ \frac{dE_{\mu_p^*}[X]}{d \Delta \lambda} \ge - \frac{d E_{\mu_q^*} [X]}{d \Delta \lambda} $$
\end{lemma}
\begin{proof}

From Lemma \ref{Lemma2} we know that

$$\frac{d}{d \Delta \lambda} \bigg[ E_{\mu_q^*} [X]^2 - E_{\mu_q^*} [X^2] \bigg] \bigg|_{\Delta \lambda = 0}= V_{+}'(0)  = 0$$

Because $V'(\Delta \lambda)$  is continuous from above at zero\footnote{
$$ \lim_{\Delta \lambda \rightarrow 0} V'(\Delta \lambda)
= \lim_{\Delta \lambda \rightarrow 0} \lim_{h \rightarrow 0} 
\frac{V(\Delta \lambda + h)}{h} 
= \lim_{h \rightarrow 0}\frac{1}{h}\lim_{\Delta \lambda \rightarrow 0 } V(\Delta \lambda + h) = 0 
$$
}
 we know that there exists a $\gamma$ such that $|\gamma| < \frac{3 - \bar{\lambda} }{\bar{\lambda}}\frac{d}{d \Delta \lambda} E_{\mu_q^*} [X]$ (note that $3-\bar{\lambda}$ and $\frac{d}{d \Delta \lambda} E_{\mu_q^*} [X]$ are both negative, so their product is positive) and

$$ V'(\Delta \lambda) = \frac{d}{d \Delta \lambda} \bigg[ E_{\mu_q^*} [X]^2 - E_{\mu_q^*} [X^2] \bigg] = \gamma,$$

for $\Delta \lambda$ small enough.  Note that $\gamma$ might be positive or negative, but for a small value of $\Delta \lambda$ we can find a small enough $\gamma$ in magnitude such that the two expected values are sufficiently close.

\begin{eqnarray}
\frac{d}{d \Delta \lambda}  E_{\mu_p^*} [X] &=& \bar{\lambda} \cdot \frac{d}{d \Delta \lambda} \bigg[ E_{\mu_q^*} [X] - E_{\mu_q^*} [X^2]  \bigg], \mbox{  from Lemma } \ref{Lemma1} \nonumber  \\
&=& \bar{\lambda} \cdot \frac{d}{d \Delta \lambda} \bigg[ E_{\mu_q^*} [X] - E_{\mu_q^*} [X]^2  \bigg] +\bar{\lambda} \gamma \nonumber \\
&=&  \bar{\lambda} \cdot \bigg[ \frac{d}{d \Delta \lambda} E_{\mu_q^*} [X] + 2 E_{\mu_q^*} [X] \bigg( -\frac{d}{d \Delta \lambda} E_{\mu_q^*} [X]   \bigg) \bigg] +  \bar{\lambda} \gamma, \mbox{  note that  }\frac{d}{d \Delta \lambda}E_{\mu_q^*} [X] < 0  \nonumber  \\
&>&  \bar{\lambda} \cdot \bigg[ \frac{d}{d \Delta \lambda} E_{\mu_q^*} [X] + 2 \bigg(\frac{\bar{\lambda}-1}{\bar{\lambda}}\bigg) \bigg( - \frac{d}{d \Delta \lambda} E_{\mu_q^*} [X]  \bigg)  \bigg] +  \bar{\lambda} \gamma,   \mbox{  by defn of } \mu_q^*, E_{\mu_q^*} [X]  > \frac{\bar{\lambda} - 1}{\bar{\lambda}}  \nonumber \\
&=& (2 - \bar{\lambda} )\cdot  \frac{d}{d \Delta \lambda} E_{\mu_q^*} [X] +  \bar{\lambda} \gamma  \nonumber \\
&>& (2 - \bar{\lambda} )\cdot  \frac{d}{d \Delta \lambda} E_{\mu_q^*} [X] +  \bar{\lambda} \bigg( \frac{\bar{\lambda}-3}{\bar{\lambda}} \frac{d}{d \Delta \lambda} E_{\mu_q^*} [X]  \bigg), \mbox{ because } \gamma > \frac{ \bar{\lambda} -3 }{\bar{\lambda}}\frac{d}{d \Delta \lambda} E_{\mu_q^*} [X] \nonumber \\
&=&  - \frac{d}{d \Delta \lambda} E_{\mu_q^*} [X]
\end{eqnarray}

\end{proof}
\noindent

\begin{proof} of Theorem \ref{PeriodTwoTheorem}. \quad 
From Lemma \ref{Lemma3} we know that $\displaystyle{\frac{d E_{\mu_q^*}[X]}{d \Delta \lambda}} < 0$.  Combining this with Lemma \ref{Lemma4}, it follows that $\displaystyle{\frac{d E_{\mu_p^*}[X]}{d \Delta \lambda}}$ is both positive and greater in magnitude than $\displaystyle{\frac{d E_{\mu_q^*}[X]}{d \Delta \lambda}}$.  Thus, for some interval around $\Delta \lambda = 0$, $E_{\mu_p^*}[X]$  is farther from $p(\bar{\lambda})$ than $E_{\mu_q^*}[X]$ is from $q(\bar{\lambda})$, implying that their mean, $E[X]$, is larger than the expected value of the deterministic map for $\lambda = \bar{\lambda}$.  \end{proof}

\section{Properties of the Deterministic Logistic Map} \label{sec:deterministic_background}

 Here we recall a few properties of the deterministic logistic map, $S_{\lambda}(x)$.  For a  more complete discussion see, for example, \cite{May, Devaney, Chaos}.
 
\begin{itemize}
\item For $\lambda$ in the interval $[0,1]$, 0 is the unique attracting fixed point of $S_{\lambda}(x)$, and, for every $x_0 \in [0,1]$, the orbit $\{ x_n \} \rightarrow 0$ as $n \rightarrow \infty$.
\item For $\lambda$ in the interval $(1, 3]$ there are two fixed points of $S_{\lambda}(x)$ in $[0,1]$: $x^0 = 0$ and $x^1 = \displaystyle{\frac{\lambda -1 }{\lambda}}$.  In this regime, zero is an unstable fixed point, and $x^1$ is attracting.  For every $x_0 \in (0,1)$, the orbit $\{ x_n \} \rightarrow x^1$, the non-zero steady state.  In other words, any initial distribution of $x$-values converges to a measure supported on a single point.
\item For  $\lambda$ in the interval $(3, 1 + \sqrt{6}]$, the two fixed points  (at zero and $\frac{\lambda -1}{\lambda}$) are unstable, and the logistic map has a stable period-$2$ cycle. Almost all initial values in $(0,1)$ have orbits that converge  to the period two orbit \cite[p.~359]{Chaos}, seen in Equations \ref{peq} and \ref{qeq}.  Note the average along the period two orbit is $\displaystyle{\frac{\lambda + 1}{2 \lambda}}$, so that the long-term average along almost every orbit is also $\displaystyle{\frac{\lambda + 1}{2 \lambda}}$.
\item For $\lambda$ in the interval $( 1 + \sqrt{6}, 3.54409 ]$ (approximately) both the fixed points and the period two cycle are  unstable.  However,  the logistic map has a stable attracting period-4 cycle for parameter values in this interval:  the population oscillates between $4$ values.   Almost all initial values of $x_0 \in (0,1)$ have orbits that converge to this period-4 cycle. Therefore, any initial continuous distribution of $x$-values will converge to a distribution supported on these four points.
\item The parameter values at which the system transitions into the stable period 2-regime and into the stable period-4 regime are called {\it bifurcation points}.  Define the parameter value $\lambda_{c_{k}}$ to be the value at which bifurcation into a stable period $k$ regime occurs.  The first few bifurcation points can be calculated as $\lambda_{c_2} = 3$ and $\lambda_{c_4} = 1 + \sqrt{6}$.  Since the period of the attracting cycle doubles at these points, the system is said to undergo a {\it period-doubling} bifurcation.
\item Period doublings occur as $\lambda $ increases.  The bifurcation parameter values get closer and closer together, yielding stable cycles of period 8, 16, 32 and so on. As the period increases, the  $\lambda$-interval between period doubling narrows.
At   $\lambda_{c_{2^\omega}} = \lim_{n \rightarrow \infty} \lambda_{c_{2^n}} \approx 3.56995$, stable periodic orbits of odd period appear in quick succession, according to Sarkovskii's theorem \cite{Devaney}, culminating in a stable period-3 orbit at $\lambda_{c_3} \approx 3.8284$.  
After the period-3 region, the deterministic logistic map exhibits chaotic behavior  \cite{Chaos,Aulbach04}.
\end{itemize}

The long-term behavior of the logistic map $S_\lambda(x)$  as $\lambda$ varies between 0 and 4 can be depicted in a {\it bifurcation diagram}, as shown in Figure~\ref{bifurc}.

\begin{figure} [H]
\begin{center}
\includegraphics[width = .8\textwidth]{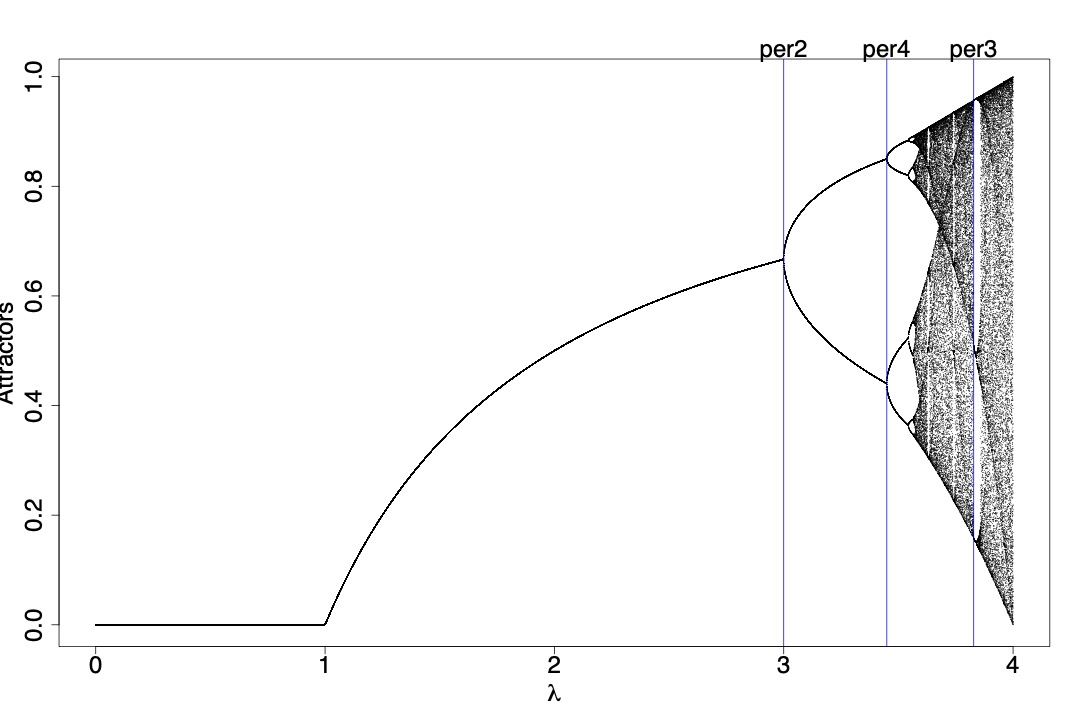}
\end{center}
\caption{Bifurcation diagram for $S_\lambda(x)$.
For each value of the parameter, $\lambda$   {  One hundred initial values of $x_0$ were generated uniformly on [0,1].  For each $x_0$ value, the logistic map was iterated 1000 times for each $\lambda$ on [0,4], in increments of 0.001.  The last iteration for each $x_0$ at each $\lambda$ is plotted.  The bifurcation values for the period 2, period 4 and period 3 bifurcations are indicated as vertical lines.}  }
\label{bifurc}
\end{figure}

\end{appendix}

\section*{Disclosure statement}
The authors report there are no competing interests to declare.

\bibliographystyle{tfs}
\bibliography{meanmean}

\end{document}